\documentclass[12pt]{amsart}
\usepackage{amssymb,amsmath,amsthm,latexsym,mathtools,dsfont}
\usepackage{mathrsfs}
\usepackage{mdwlist}
\usepackage{graphicx}
\usepackage{enumerate}
\usepackage[shortlabels]{enumitem}
\usepackage{a4wide}
\input xy
\xyoption{all}
\usepackage{physics}
\usepackage{color}
\definecolor{amber}{rgb}{1.0, 0.75, 0.0}

\newcommand{\R}{\mathbb{R}}

\newcommand{\N}{\mathbb{N}}
\newcommand{\Z}{\mathbb{Z}}
\newcommand{\C}{\mathbb{C}}
\newcommand{\Ma}{\mathbb{M}}

\newcommand{\Aa}{\mathcal{A}}
\newcommand{\Bb}{\mathcal{B}}

\newcommand{\Hh}{\mathcal{H}}

\newcommand{\LL}{\mathscr{L}}

\newcommand{\ii}{\mathrm{i}}
\newcommand{\ind}{\mathds{1}}
\newcommand{\p}{\varphi}
\newcommand{\e}{\varepsilon}
\newcommand{\w}{\widetilde}
\newcommand{\oo}{\overline}
\newcommand{\MM}{\mathcal{M}}
\newcommand{\wpsi}{%
  \mspace{2mu}%
  \widetilde{\mspace{-2mu}\rule{0pt}{1.3ex}\smash[t]{\psi}}%
}

\newcommand{\n}[1]{\|#1\|}
\newcommand{\nn}[1]{{\vert\kern-0.25ex\vert\kern-0.25ex\vert #1 
    \vert\kern-0.25ex\vert\kern-0.25ex\vert}}
\newcommand{\lnn}[1]{{\left\vert\kern-0.25ex\left\vert\kern-0.25ex\left\vert #1 
    \right\vert\kern-0.25ex\right\vert\kern-0.25ex\right\vert}}

\newcommand{\dist}{\mathrm{dist}}

\newcommand{\ccup}{\scalebox{0.85}{$\bigcup$}}
\newcommand{\dast}{{\ast\ast}}
\newcommand{\dprime}{{\prime\prime}}
\renewcommand{\leq}{\leqslant}
\renewcommand{\geq}{\geqslant}
\newcommand{\cs}{{\rm C}$^\ast$}
\newcommand{\sa}{\mathrm{sa}}
\newcommand{\sss}{\mathrm{s}}
\newcommand{\rrr}{\mathrm{r}}

\renewcommand{\arraystretch}{1.3}

\newcommand{\dpr}{\mathsf{d.p.}}
\newcommand{\edpr}[1]{#1\mbox{-}\mathsf{d.p.}}

\newcommand{\eoz}[1]{#1\mbox{-}\mathsf{o.z.}}
\newcommand{\esa}[1]{#1\mbox{-}\mathsf{s.a.}}
\newcommand{\eJh}[1]{#1\mbox{-}\mathsf{J.h.}}

\newtheorem{theorem}{Theorem}[section]
\newtheorem{lemma}[theorem]{Lemma}
\newtheorem{proposition}[theorem]{Proposition}
\newtheorem{corollary}[theorem]{Corollary}
\newtheorem*{theoremA}{Theorem A}
\newtheorem*{theoremB}{Theorem B}

\theoremstyle{definition}
\newtheorem{definition}[theorem]{Definition}

\theoremstyle{remark}
\newtheorem*{remark}{Remark}

\numberwithin{equation}{section}

\title{Approximately order zero maps between C$^\ast$-algebras}
\author{Tomasz Kochanek}
\address{Institute of Mathematics, Polish Academy of Sciences, \'Sniadeckich 8, 00-656 Warsaw, Poland\, {\rm and}\, Institute of Mathematics, University of Warsaw, Banacha~2, 02-097 Warsaw, Poland}
\email{tkoch@impan.pl}
\thanks{This work was supported by the GA\v{C}R project 18-00960Y}

\begin{document}
\begin{abstract}
We investigate linear operators between \cs-algebras which approximately preserve involution and orthogonality, the latter meaning that for some $\e>0$ we have $\n{\phi(x)\phi(y)}\leq\e\n{x}\n{y}$ for all positive $x,y$ with $xy=0$. We establish some structural properties of such maps concerning approximate Jordan-like equations and almost commutation relations. In some situations (e.g. when the codomain is finite-dimensional), we show that $\phi$ can be approximated 
by an~approximate Jordan $^\ast$-homomorphism, with both errors depending only on $\n{\phi}$ and $\e$. 
\end{abstract}
\maketitle
\tableofcontents

\section{Introduction}
\noindent
There is a widely developed theory concerning the question to what extent the zero-product structure determines the whole structure of a~given Banach algebra, or to what extent the action on zero-product elements characterizes homomorphisms, derivations etc. For example, \cite{alaminos_studia} contains a~series of results characterizing zero-product-preserving maps and dealing with the question whether such maps must be automatically weighted homomorphisms. In \cite{alaminos_b}, it is shown that a~certain generalized zero-product-preserving property force the map in question to be a~homomorphism.

In the setting of \cs-algebras, an important role is played by linear operators which preserve orthogonality or, equivalently, preserve zero products of self-adjoint elements. Such maps are usually assumed to be completely positive, as well-behaved amplifications to matrix algebras are quite useful in noncommutative topology. Winter and Zacharias \cite{WZ1} called those maps {\it order zero} and they exhibited their importance as `building blocks' of noncommutative partitions of unity. Consequently, order zero maps proved to be the~key ingredient to define {\it nuclear dimension} of \cs-algebras, a~noncommutative analogue of the covering dimension (see \cite{WZ2} and the references therein). 

In this paper, we deal with approximately order zero operators between \cs-algebras, a~notion somewhat analogous to the notion of approximately multiplicative maps which were profoundly investigated by B.E.~Johnson in a~series of his paper (see, e.g., \cite{johnson_f} and \cite{johnson}). Favoring the $^\ast$-algebra structure over the order structure, we do not assume complete positivity, instead we deal with operators which simultaneously preserve orthogonality and involution.

Recall that the usual relation of orthogonality $x\perp y$, for $x,y$ from a~given \cs-algebra, is defined by the condition $xy=yx=x^\ast y=xy^\ast=0$. Note that for self-adjoint $x,y$ the condition $x\perp y$ is equivalent to $xy=0$. For a~\cs-algebra $\Aa$, we denote by $\Aa_\sa$ and $\Aa_+$ the sets of all self-adjoint and all positive elements of $\Aa$, respectively.
\begin{definition}
Let $\Aa$, $\Bb$ be \cs-algebras, $\phi\colon \Aa\to \Bb$ a~bounded linear operator and $\e\geq 0$. We say that $\phi$ is an~$\e$-{\it order zero map} ($\eoz{\e}$ for short) if it satisfies the~condition
\begin{equation}\label{oz_def}
x,y\in \Aa_+,\,\, x\perp y\,\,\xRightarrow[]{\phantom{xxx}}\,\, \n{\phi(x)\phi(y)}\leq \e\n{x}\n{y}.
\end{equation}
We say $\phi$ is $\e$-{\it self-adjoint} ($\esa{\e}$ for short) if it satisfies
$$
\n{\phi(x^\ast)-\phi(x)^\ast}\leq\e\n{x}\quad\mbox{for every }x\in \Aa.
$$
Finally, we call $\phi$ an~$\e$-{\it disjointness preserving map} ($\edpr{\e}$ for short), provided it is both $\eoz{\e}$ and $\esa{\e}$ Order zero maps, self-adjoints maps and disjointness preserving maps are defined as above with $\e=0$.
\end{definition}

It is worth mentioning that the stability problem for almost disjointness preserving operators between $C(X)$-spaces was first considered by Dolinar \cite{dolinar} and then completely solved in a~series of papers by Araujo and Font (\cite{araujo_JLMS}, \cite{araujo_JAT}, \cite{araujo_PAMS}). Recently, almost disjointness preserving operators on Banach lattices were studied by Oikhberg and Tradacete \cite{oikhberg}.

Below, we recall two important results which characterize operators preserving orthogonality on \cs-algebras. The first one says, roughly, that self-adjoint maps of this type are compressions of Jordan $^\ast$-homomorphisms, whereas the second one says that completely positive maps of this type are compressions of $^\ast$-homomorphisms. For more results characterizing disjointness preserving maps, the reader may consult \cite{chebotar} and \cite{font}.

Wolff \cite{wolff} defined a~bounded linear operator $T\colon\Aa\to\Bb$ to be {\it disjointness preserving}, provided that it is self-adjoint, i.e. $T(x^\ast)=T(x)^\ast$ for every $x\in\Aa$, and $T(x)T(y)=0$ for all $x,y\in\Aa_\sa$ with $xy=0$. Winter and Zacharias \cite{WZ1} defined a~c.p. map $\p\colon \Aa\to\Bb$ to have {\it order zero}, provided that $\p(x)\perp\p(y)$ for all $x,y\in\Aa_+$ with $xy=0$. In fact, such a~map must preserve orthogonality of all elements of $\Aa$ (see \cite[Remark~2.4]{WZ1}).

\begin{theorem}[{\cite[Thm. 2.3]{wolff}}]\label{wolff_thm}
Let $\Aa$ and $\Bb$ be \cs-algebras with $\Aa$ unital. Let $T\colon\Aa\to\Bb$ be a~disjointness preserving operator with $h\coloneqq T(1_\Aa)$. Then we have:
\begin{itemize}[leftmargin=32pt]
    \item[{\rm (a)}] $T(\Aa)\subseteq \mathcal{C}\coloneqq\oo{h\{h\}^\prime}$; 
    \item[{\rm (b)}] there exists a Jordan $^\ast$-homomorphism $S$ mapping $\Aa$ into the multiplier algebra $\mathscr{M}(\mathcal{C})$ such that $S(1_\Aa)=1_{\mathscr{M}(\mathcal{C})}$ and $T(x)=hS(x)$ for every $x\in\Aa$.
\end{itemize}
\end{theorem}

\begin{theorem}[{\cite[Thm. 3.3]{WZ1}}]\label{WZ_thm}
Let $\Aa$ and $\Bb$ be \cs-algebras and $\p\colon\Aa\to\Bb$ a~c.p. order zero map with $\mathcal{C}\coloneqq \mathrm{C}^\ast(\p(\Aa))$. Then, there exists a~positive element $h\in\mathscr{M}(\mathcal{C})\cap\mathcal{C}^\prime$ with $\n{h}=\n{\p}$ and a~$^\ast$-homomorphism $\pi\colon\Aa\to \mathscr{M}(\mathcal{C})\cap\{h\}^\prime$ such that $\p(x)=h\pi(x)$ for every $x\in\Aa$. Moreover, if $\Aa$ is unital, then $h=\p(1_\Aa)$.
\end{theorem}

The study of almost zero-product-preserving, or almost orthogonality preserving maps can be regarded as a~part of the general Ulam's stability problem \cite{ulam}. One deep theorem in this context was given by Alaminos, Extremera and Villena \cite{alaminos} who showed that in many cases an~almost zero-product-preserving operator $T$ must be close to a~compression of a~homomorphism. However, a~crucial assumption was that said map is surjective, in which case one can consider its {\it openness index} defined by
$$
\mathsf{op}(T)=\inf\bigl\{M>0\colon \mbox{for every }x\in X\mbox{ there is }y\in Y\mbox{ with }T(x)=y\mbox{ and }\n{x}\leq M\n{y}\bigr\}.
$$
(By the Open Mapping Theorem, such a~constant is finite whenever $T$ maps a~Banach space onto a~Banach space.)
\begin{theorem}[{\cite[Thm.~4.7]{alaminos}}]\label{thm_alaminos}
Let $\Aa$ be either the group algebra $L^1(G)$ for a~locally compact group $G$ or a~\cs-algebra and let $\Bb$ be a~Banach algebra. Assume both $\Aa$ and $\Bb$ are amenable and that for some Banach $\Bb$-bimodule $X$ the multiplier algebra $\mathscr{M}(\Bb)$ is isomorphic as a~Banach $\Bb$-bimodule to $X^\ast$. For all $\e, K, M>0$ there exists $\delta=\delta(\e,K,M)>0$ such that for every surjective operator $T\in\LL(\Aa,\Bb)$ satisfying:
\begin{itemize}[leftmargin=32pt]
\item[{\rm (i)}] $\n{T}\leq K$, 
\item[{\rm (ii)}] $\mathsf{op}(T)\leq M$,
\item[{\rm (iii)}] $\n{T(x)T(y)}\leq\delta\n{x}\n{y}$ for all $x,y\in A$ with $xy=0$,
\end{itemize}
there exist an~invertible element $\nu$ in the cener $\mathcal{Z}(\MM(\Bb))$ and a~continuous epimorphism $\Phi\colon \Aa\to \Bb$ with $\n{T-\nu\Phi}\leq\e$.
\end{theorem}

At this point, let us stress that there is an~essential difference between assuming that the product of values vanishes for all pairs with zero product or merely for those which are orthogonal in the \cs-algebraical sense. Indeed, it follows from another result by Alaminos, Extremera and Villena \cite{alaminos} (Theorem~\ref{alaminos_thm} below) that zero-product-preserving maps must satisfy a~multiplicativity-like property---in the unital case they must be simply multiplicative. On the other hand, by Wolff's Theorem~\ref{wolff_thm}, for (unital) $\dpr$ maps we cannot go further than Jordan $^\ast$-homomorphisms.

In view of Theorem~\ref{wolff_thm}, given an~$\edpr{\e}$ map $\phi$ we should expect that it can be approximated by a~compression of a~Jordan $^\ast$-homomorphism. Hence, a~natural question is whether $\phi$ itself is `almost' a~compressed Jordan $^\ast$-homomorphism, and a~positive answer is provided by Proposition~\ref{P_almostJordan}. A~large part of the present paper is, however, devoted to the problem of approximating $\phi$ by genuine (unital) almost Jordan $^\ast$-homomorphisms.
\begin{definition}
Let $\Aa$, $\Bb$ be \cs-algebras, $\phi\colon \Aa\to \Bb$ a~bounded linear operator and $\e\geq 0$. We say that $\phi$ is an~$\e$-{\it Jordan $^\ast$-homomorphism} ($\eJh{\e}$ for short) if it is $\esa{\e}$ and satisfies
\begin{equation*}
\n{\phi(x)^2-\phi(x^2)}\leq \e\n{x}^2\quad\mbox{for every }x\in\Aa.
\end{equation*}
\end{definition}
\noindent
Our main goal is to show that, in some natural situations, every $\edpr{\e}$ map can be approximated to within $\delta(\e)$ by a~map which behaves like an~$\eJh{\eta(\e)}$, where both $\delta(\e)$ and $\eta(\e)$ converge to zero as $\e\to 0$. Therefore, the stability problem for $\edpr{\e}$ maps gets reduced to the stability problem for $\eJh{\delta}$ maps. We believe the latter one can be attacked by similar cohomological methods as those introduced by B.E.~Johnson (\cite{johnson_memoir}, \cite{johnson_d}) and applied to almost multiplicative maps (\cite{johnson_f}, \cite{johnson}).

We use standard notation. By $\LL(\Aa,\Bb)$ we denote the space of all bounded linear operators between given \cs-algebras $\Aa$ and $\Bb$. In the second dual $\Aa^\dast$ we consider the usual (Arens) multiplication which extends the multiplication in $\Aa$ and makes $\Aa^\dast$ a~von Neumann algebra (see, e.g., \cite[\S III.5.2]{blackadar}).

In Theorems A and B below we summarize main results of this paper. The first one is more of a~structural nature; it follows from Lemma~\ref{L_unitization}, Proposition~\ref{P_almostJordan}, Lemma~\ref{close_to_hereditary} and a~combination of Proposition~\ref{alg_prop} with Corollary~\ref{h_vs_psi}. Theorem~B is a~part of Theorem~\ref{L_decomposition} and Corollary~\ref{main_corollary}. 
\begin{theoremA}
Let $\Aa$ and $\Bb$ be \cs-algebras and $\phi\in\LL(\Aa,\Bb)$ be an~$\eoz{\e}$ map. If $\Aa$ is nonunital and $\pi$ is a~nondegenerate representation of $\Bb$ on a~Hilbert space, then $\phi$ can be extended to an~$\eoz{\e}$ map $\phi^\dag\in\LL(\Aa^\dag,\pi(\Bb)^\dprime)$ defined on the unitization $\Aa^\dag$ of $\Aa$.

Assuming that $\Aa$ is unital and $h\coloneqq \phi(1_\Aa)$, the following assertions hold true: 
\begin{itemize}[leftmargin=32pt]
\setlength{\itemsep}{6pt}
\item[{\rm (a)}] We have $$\n{\phi(x)^2-h\phi(x^2)}\leq 108\e\n{x}\quad\mbox{for every }x\in\Aa.$$

\item[{\rm (b)}] There is an absolute constant $K<\infty$ such that if $\phi$ is self-adjoint, then its range lies close to the~hereditary subalgebra $\oo{h\Bb h}$ in the sense that
$$
\dist\big(\phi(x),\oo{h\Bb h}\big)\leq K\n{\phi}^{3/5}\e^{1/5}\n{x}\quad\mbox{for every }x\in\Aa.
$$

\item[{\rm (c)}] If $\phi$ is self-adjoint and $h\neq 0$ is an~algebraic element of $\Bb$, then either $$\n{\phi}\leq \sqrt{(K+2)^5\e},$$
or the range of $\phi$ lies close to the commutant $\{h\}^\prime$ in the sense that for every complex polynomial $P\in\C[z]$ with $P(0)=0$ and each $x\in\Aa$ we have
$$
\n{[P(h),\phi(x)]}\leq \frac{C\e}{\n{h}}\sup_{\abs{z}=\n{h}}\abs{P(z)}\cdot\n{x},
$$
where $C>0$ depends only on the degree of algebraicity of $h$.
\end{itemize}
\end{theoremA}

\begin{theoremB}
Let $\Aa$, $\Bb$ be \cs-algebras with $\Aa$ unital and let $\pi$ be a~nondegenerate representation of $\Bb$ on a~Hilbert space $\Hh$. Let also $\phi\in\LL(\Aa,\Bb)$ be a~self-adjoint $\eoz{\e}$ map with some $\e\in (0,1]$ and with $h\coloneqq \psi(1_\Aa)$ being an~algebraic element of $\Bb$. Then, there exists a decomposition $\phi=\phi_\sss+\phi_\rrr$, where the operators $\phi_{\sss}, \phi_{\rrr}\in\LL(\Aa,\pi(\Bb)^\dprime)$ satisfy the~following conditions:

\vspace*{1mm}\noindent
\begin{itemize}[leftmargin=32pt]
\setlength{\itemsep}{4pt}
\item[{\rm (i)}] $\displaystyle{
\n{\phi_\sss}\leq (6K+7)\n{\phi}^{4/5}\e^{1/16}}$;

\item[{\rm (ii)}] $\phi_\rrr$ takes values in a~corner subalgebra $\mathcal{C}$ of $\oo{h\pi(\Bb)^\dprime h}$;

\item[{\rm (iii)}] either $\phi_\rrr=0$ or $\phi_\rrr(1_\Aa)$ is invertible in $\mathcal{C}$, in which case $\phi_\rrr(1_\Aa)^{-1}\phi_\rrr(\,\cdot\,)$ is a~unital $\eJh{\delta}$ map with
$$
\delta=24\big(C^2(K+2)^5+10C+17\big)\n{\phi}\e^{1/16},
$$
where, again, $C>0$ depends only on the degree of algebraicity of $h$.
\end{itemize}
\noindent
In particular, if $\phi\colon\Aa\to\Ma_n(\C)$ is a~positive $\eoz{\e}$ map, then there exists a~corner subalgebra $\mathcal{C}$ of $\Ma_n(\C)$ and an~operator $\Phi\in\LL(\Aa,\mathcal{C})$ satisfying 
$$
\n{\phi-\Phi}\leq 37\n{\phi}^{4/5}\e^{1/16}
$$
and such that either $\Phi=0$ or $\Phi(1_{\Aa})$ is invertible in $\mathcal{C}$, in which case the~operator $\Phi(1_{\Aa})^{-1}\Phi(\,\cdot\,)$ is $\eJh{\delta}$ with
$$
\delta=\mathcal{O}\big(256^{n}\big)\n{\phi}\e^{1/16}\quad\mbox{as }n\to\infty.
$$
\end{theoremB}

\section{Preliminaries}
\noindent
In this section, we record a~few simple technical observations which will be useful in the sequel. However, we should start with quoting a~deep result by Alaminos, Extremera and Villena \cite{alaminos} which says, roughly speaking, that almost zero-product-preserving maps on \cs-algebras must satisfy an~approximate version of a~multiplicativity-like property. They introduced an~error function defined by the formula
$$
\zeta(s)=\left\{\begin{array}{cl}
\xi\left(\displaystyle{\frac{8\pi}{\sqrt{3\big((\frac{17^2}{3}+1)s^{-1/4}-1\big)}-1}}\right) & \mbox{ if }s>0\\
0 & \mbox{ if }s=0,
\end{array}\right.
$$
where $\xi(s)=A(s)+B(s)+\Gamma(s)$ with:
$$
A(s)=\frac{1}{2\pi}\big|2\sin(s)+s(1-\cos(s))\big|+\frac{1}{\pi}\Big|s+2\Big(\frac{1-\cos(s)}{s}\Big)\cos(s)\Bigr|,
$$
$$
B(s)=\Big|\frac{1-e^{\ii s}}{s}\Big|\sum_{k\in\Z, k\neq 0,1}\frac{\abs{1-e^{\ii ks}}}{\pi k^2},\quad \Gamma(s)=\sum_{k\in\Z, k\neq 0,1}\frac{\abs{\sin (1-k)s}}{\pi\abs{k(k-1)}}.
$$

\begin{theorem}[{\cite[Thm. 3.5]{alaminos}}]\label{alaminos_thm}
Let $\Aa$ be a \cs-algebra, $X$ a~Banach space and $\Phi\colon \Aa\times \Aa\to X$ a~bounded bilinear map satisfying
$$
x,y\in \Aa,\,\, xy=0\,\,\xRightarrow[]{\phantom{xxx}}\,\, \n{\Phi(x,y)}\leq \e\n{x}\n{y}
$$
with some $\e\geq 0$. Let also $K$ be any number satisfying $K\geq\n{\Phi}$ and $K>\e$. Then
\begin{equation*}
\begin{split}
\n{\Phi(xy,z)-&\Phi(x,yz)}\\
&\leq\Big[\Big(\frac{17^2}{3}+1\Big)^2K^{1/2}\e^{1/2}\Big(2+\zeta\Big(\frac{\e}{K}\Big)\Big)+K\zeta\Big(\frac{\e}{K}\Big)\Bigr]\n{x}\n{y}\n{z}
\end{split}
\end{equation*}
for all $x,y,z\in \Aa$.
\end{theorem}
\noindent
Putting $\e=0$ and $\Phi(x,y)=\phi(x)\phi(y)$, for a~given map $\phi\in\LL(\Aa,\Bb)$, where $\Bb$ is a~Banach algebra, we see that Theorem~\ref{alaminos_thm} yields a~characterization of zero-product-preserving maps on \cs-algebras. In general, it reduces the study of almost zero-product-preserving maps to the study of stability of the equation $\phi(xy)\phi(z)=\phi(x)\phi(yz)$. In contrast, the main goal of this paper is to reduce the study of almost order zero (almost disjointness preserving) maps to the study of almost Jordan homomorphisms, that is, stability of the equation $\phi(x)^2=\phi(x^2)$.

It will be quite helpful for us to know the~asymptotic behavior of the~error function $\zeta$.

\begin{lemma}\label{zeta_O}
We have $\zeta(s)=\mathcal{O}(s^{1/16})$ as $s\to 0^{+}$.
\end{lemma}
\begin{proof}
By elementary trigonometry, we have $A(s)=\mathcal{O}(s)$ as $s\to 0^+$. We shall prove that $B(s)=\mathcal{O}(s^{1/2})$ and $\Gamma(s)=\mathcal{O}(s^{1/2})$. The~result will then follow, since $\zeta$ can be written as a~composition $\xi\circ \alpha$, where
$$
\frac{\alpha(s)}{s^{1/8}}=\frac{C_1}{\sqrt{C_2-3s^{1/4}}-s^{1/8}}\xrightarrow[\,\,s\to 0+\,]{}\frac{C_1}{\sqrt{C_2}}
$$
with suitable constants $C_1,C_2>0$.

For a moment, fix any $M>0$ and $0<s\leq \pi/M$. Notice that for each $k\in\Z$ with $\abs{k}\leq M$ we have 
$$
\abs{1-e^{\ii ks}}\leq\abs{1-e^{\ii Ms}}=\sqrt{2(1-\cos Ms)}=Ms\sqrt{2\,\Big(\frac{1}{2!}-\frac{(Ms)^2}{4!}+\frac{(Ms)^4}{6!}-\ldots\Big)}<Ms.
$$
Therefore,
\begin{equation*}
\begin{split}
\sum_{k\in\Z, k\neq 0,1}\frac{\abs{1-e^{\ii ks}}}{\pi k^2} &=\sum_{\abs{k}>M}+\sum_{\abs{k}\leq M, k\neq 0,1}\\
&<\frac{4}{\pi}\!\sum_{k=M+1}^\infty\frac{1}{k^2}+\frac{Ms}{\pi}\!\!\sum_{\abs{k}\leq M, k\neq 0}\frac{1}{k^2}<\frac{4}{\pi M}+\frac{\pi Ms}{3}.
\end{split}
\end{equation*}
Thus, putting $M\coloneqq s^{-1/2}$ with $s\to 0^+$ we see that $B(s)=\mathcal{O}(s^{1/2})$.

Again, fix any $M>0$ and note that for $k\in\Z$, $\abs{k}<M$ we have $\abs{\sin (1-k)s}\leq\abs{(1-k)s}\leq Ms$. Hence, with some absolute constant $C>0$, we have
\begin{equation*}
\begin{split}
\Gamma(s) &=\sum_{k\in\Z, k\neq 0,1}\frac{\abs{\sin (1-k)s}}{\pi\abs{k(k-1)}}=\sum_{\abs{k}\geq M}+\sum_{\abs{k}<M, k\neq 0,1}\\
&<\frac{2}{\pi}\sum_{k=M}^\infty\frac{1}{(k-1)k}+CMs=\frac{2}{\pi(M-1)}+CMs.
\end{split}
\end{equation*}
Putting $M\coloneqq s^{-1/2}$ as above we obtain $\Gamma(s)=\mathcal{O}(s^{1/2})$.
\end{proof}

\begin{lemma}\label{L_almostsa}
Let $\phi\colon \Aa\to \Bb$ be an $\esa{\delta}$ map between \cs-algebras $\Aa$ and $\Bb$. Then there exists a~self-adjoint map $\psi\colon \Aa\to \Bb$ such that $\n{\phi-\psi}\leq\frac{1}{2}\delta$. Moreover, if $\phi$ is $\eoz{\e}$, then $\psi$ can be picked to be $\eoz{(\e+\frac{1}{2}\delta\n{\phi})}$
\end{lemma}
\begin{proof}
Consider the standard involution $\phi\mapsto\phi^\ast$ in $\LL(\Aa,\Bb)$ given by $\phi^\ast(x)=\phi(x^\ast)^\ast$ and define $\psi=\frac{1}{2}(\phi+\phi^\ast)$. Plainly, $\psi$ is self-adjoint and satisfies the desired inequality.

Now, if $\phi$ is $\eoz{\e}$, then for all $x,y\in \Aa_+$ with $xy=0$ we have
$$
\n{\phi(x)\phi(y)^\ast}\leq\n{\phi(x)\phi(y)}+\n{\phi(x)(\phi(y)^\ast-\phi(y))}\leq (\e+\delta\n{\phi})\n{x}\n{y}
$$
and, of course, the same estimate is valid for $\n{\phi(x)^\ast\phi(y)}$. Hence,
\begin{equation*}
\n{\psi(x)\psi(y)}=\frac{1}{4}\n{\phi(x)\phi(y)+\phi(x)\phi(y)^\ast+\phi(x)^\ast\phi(y)+\phi(x)^\ast\phi(y)^\ast}\leq \frac{1}{4}(4\e+2\delta\n{\phi})\n{x}\n{y},
\end{equation*}
as desired.
\end{proof}

\begin{lemma}\label{L_eoz}
Let $\phi\colon \Aa\to \Bb$ be an $\eoz{\e}$ map between \cs-algebras $\Aa$ and $\Bb$. Then:
\begin{itemize}
\setlength{\itemsep}{4pt}
\item[{\rm (a)}] for all $x,y\in \Aa_\sa$ with $x\perp y$ we have $\n{\phi(x)\phi(y)}\leq 4\e\n{x}\n{y}$;

\item[{\rm (b)}] for all $x,y\in \Aa$ with $x\perp y$ we have $\n{\phi(x)\phi(y)}\leq 16\e\n{x}\n{y}$;

\item[{\rm (c)}] if $\phi$ is a completely positive contraction, then 
$$
\n{\phi(x)\phi(y)}\leq \e^{1/2}\n{x}\n{y}\quad\mbox{for all }x,y\in\Aa\,\mbox{ with }\, x\perp y.
$$
\end{itemize}
\end{lemma}
\begin{proof}
(a) Fix $x,y\in \Aa_\sa$, $x\perp y$ and let $x=x_1-x_2$, $y=y_1-y_2$ be the~Jordan decompositions of $x$ and $y$, that is, $x_1,x_2,y_1,y_2\in \Aa_+$, $x_1x_2=0$ and $y_1y_2=0$. These elements are defined by functional calculus on \cs$(x,y)$, namely, $x_1=f(x)$, $x_2=g(x)$, $y_1=f(y)$ and $y_2=g(y)$, where $f(t)=\max\{t,0\}$ and $g(t)=-\min\{t,0\}$. Since $xy=0$, we have $x_iy_j=0$ and hence $\n{\phi(x_i)\phi(y_j)}\leq\e\n{x_i}\n{y_j}$ for all $1\leq i,j\leq 2$. Therefore,
\begin{equation*}
\begin{split}
\n{\phi(x)\phi(y)} &=\n{(\phi(x_1)-\phi(x_2))(\phi(y_1)-\phi(y_2))}\\
&\leq\sum_{i,j=1}^2\n{\phi(x_i)\phi(y_j)}\leq\e\sum_{i,j=1}^2\n{x_i}\n{y_j}\leq 4\e\n{x}\n{y}.
\end{split}
\end{equation*}

\noindent
(b) We simply decompose $x$ and $y$ into real and imaginary parts: $x=x_1+\ii x_2$, $y=y_1+\ii y_2$. where $x_1=\frac{1}{2}(x+x^\ast)$, $x_2=\frac{1}{2\ii}(x-x^\ast)$, $y_1=\frac{1}{2}(y+y^\ast)$, $y_2=\frac{1}{2\ii}(y-y^\ast)$. By the assumption that $x\perp y$, we have $x_iy_j=0$ for all $1\leq i,j\leq 2$. Of course, $\n{x_i}\n{y_j}\leq\n{x}\n{y}$ and hence assertion (a) yields
\begin{equation*}
\begin{split}
\n{\phi(x)\phi(y)} &=\n{\phi(x_1)\phi(y_1)-\phi(x_2)\phi(y_2)+\ii\phi(x_1)\phi(y_2)+\ii\phi(x_2)\phi(y_1)}\\
&\leq \sum_{i,j=1}^2\n{\phi(x_i)\phi(y_j)}\leq 4\e\sum_{i,j=1}^2\n{x_i}\n{y_j}\leq 16\e\n{x}\n{y}.
\end{split}
\end{equation*}

\noindent
(c) First, recall that for any $x,y\in\Aa$ we have $x\perp y$ if and only if the~following four orthogonality conditions hold: $x^\ast x\perp y^\ast y$, $x^\ast x\perp yy^\ast$, $xx^\ast\perp y^\ast y$ and $xx^\ast\perp yy^\ast$. 

Fix any elements $x,y\in\Aa$ with $x\perp y$ and $\n{x},\n{y}\leq 1$. Due to the remark above, we have $\n{\phi(y^\ast y)\phi(x^\ast x)}\leq \e$ and, in view of Kadison's inequality (see \cite[\S II.6.9.14]{blackadar}), $\phi(x)^\ast\phi(x)\leq \phi(x^\ast x)$ and $\phi(y)\phi(y)^\ast\leq \phi(y^\ast y)$. Therefore,

\begin{equation*}
    \begin{split}
        \n{\phi(x)\phi(y)}^4 &=\n{\phi(y)^\ast\phi(x)^\ast\phi(x)\phi(y)}^2\leq \n{\phi(y)^\ast\phi(x^\ast x)\phi(y)}^2\\
        &=\n{\phi(x^\ast x)^{1/2}\phi(y)\phi(y)^\ast\phi(x^\ast x)^{1/2}}^2\leq \n{\phi(x^\ast x)^{1/2}\phi(y^\ast y)\phi(x^\ast x)^{1/2}}^2\\
        &=\n{\phi(x^\ast x)^{1/2}\phi(y^\ast y)\phi(x^\ast x)\phi(y^\ast y)\phi(x^\ast x)^{1/2}},
    \end{split}
\end{equation*}
where $z\coloneqq \phi(x^\ast x)^{1/2}\phi(y^\ast y)\phi(x^\ast x)\phi(y^\ast y)\phi(x^\ast x)^{1/2}\in\Bb_\sa$. Observe that for each $n\in\N$, we have 
$$
\n{z^n}\leq \n{\phi(x^\ast x)}\cdot\n{\phi(y^\ast y)\phi(x^\ast x)}^{2n-1}\leq \e^{2n-1}.
$$
The norm of $z$, being equal to its spectral radius, is then estimated by 
\begin{equation*}
\n{z}\leq \lim_{n\to \infty}\e^{(2n-1)/n}=\e^2.\qedhere
\end{equation*}
\end{proof}

\begin{remark}
The above argument is very similar to the one used in the proof of \cite[Prop.~3.1]{WZ2}. Of course, one can obtain the same estimates for $\phi(y)\phi(x)$, $\phi(x)^\ast \phi(y)$ and $\phi(x)\phi(y)^\ast$. Hence, $\eoz{\e}$ c.p.c. maps send orthogonal elements to `$\e^{1/2}$-orthogonal' ones. 
\end{remark}

\section{An extension result and approximate Jordan equations}
\noindent
For a nonunital \cs-algebra $\Aa$ we denote by $\Aa^\dag$ its one-point unitization, i.e. $\Aa^\dag=\Aa\oplus\C$ as a~vector space, where $\Aa$ forms a~closed ideal of $\Aa^\dag$ of codimension one. Recall that $\Aa^\dag$ is equipped with the~{\it operator norm} (regarding elements of $\Aa\oplus\C$ as left multiplication operators on $\Aa$) defined by
$$
\n{(x,\alpha)}_{\mathrm{op}}=\sup\bigl\{\n{xy+\alpha y}\colon y\in A,\, \n{y}\leq 1\bigr\}\quad (x\in \Aa,\,\alpha\in\C)
$$
and which satisfies the \cs-condition. The $\ell_1$-norm on $\Aa\oplus\C$, although not being a~\cs-norm, happens to be equivalent to the~operator norm. Indeed, as was shown by Gaur and Kov\'a\v{r}\'{i}k \cite{GK}, we have
$$
\n{x}+\abs{\alpha}\leq 3\n{(x,\alpha)}_{\mathrm{op}}\quad\mbox{for all }x\in \Aa_\sa,\, \alpha\in\C
$$
and the constant $3$ is sharp.

Recall that, by the~von~Neumann Bicommutant Theorem, if $\mathcal{M}$ is a~$^\ast$-algebra acting nondegenerately on a~Hilbert space $\Hh$, then $\mathcal{M}^{\dprime}$ coincides with the~closure of $\mathcal{M}$ with respect to the~weak (equivalently, strong) operator topology (see, e.g., \cite[\S II.3]{takesaki}). Note also that on bounded sets the weak topology coincides with the~$\sigma$-weak topology which is the same as the~weak$^\ast$ topology on $\LL(\Hh)$ generated by its~canonical predual, the space of trace-class operators (see \cite[Lemma~II.2.5]{takesaki}).

\begin{lemma}\label{L_unitization}
Let $\Aa$, $\Bb$ be \cs-algebras, $\Aa$ be nonunital, and let $\pi$ be a~nondegenerate representation of $\Bb$ on a~Hilbert space $\Hh$. Then, for every $\eoz{\e}$ operator $\phi\colon \Aa\to \Bb$ there exists an~$\eoz{\e}$ operator $\phi^\dag\colon \Aa^\dag\to \pi(\Bb)^\dprime$ which extends $\phi$ so that the following diagram commutes:
$$
\xymatrix{
\Aa\ar[r]^\phi \ar@{^{(}->}[d] & \Bb\ar[r]^\pi & \pi(\Bb)\ar@{^{(}->}[r] & \pi(\Bb)^\dprime\subseteq\LL(\Hh)\\
\Aa^\dag\ar@{-->}[rrru]^<<<<<<<<<<<<{\phi^\dag} & & &
}
$$
i.e. $\phi^\dag(x)=\pi(\phi(x))$ for every $x\in \Aa$. Moreover, if $\phi$ is $\esa{\delta}$, then we can pick $\phi^\dag$ to be $\esa{6\delta}$ If $\phi$ is completely positive, then $\phi^\dag$ can be completely positive as well.
\end{lemma}
\begin{proof}
Fix a bounded approximate unit $(u_\lambda)_{\lambda\in\Lambda}$ of $\Aa$. By passing to a~subnet and using the~Banach--Alaoglu theorem (or the~$\mathrm{w.o.t.}$-compactness of the unit ball of $\LL(\Hh)$) we may assume that there exists a~limit
$$
z_0\coloneqq \mathrm{w.o.t.}\lim_{\lambda} \pi\circ\phi(u_\lambda).
$$
Define $\phi^\dag\colon \Aa^\dag\to\pi(\Bb)^\dprime$ by the formula
$$
\phi^\dag(x+\alpha\cdot 1_{A^\dag})=\pi\circ\phi(x)+\alpha z_0\quad (x\in \Aa,\,\alpha\in\C).
$$
Note that $\phi^\dag$ does take values in $\pi(\Bb)^\dprime=\overline{\pi(\Bb)}^{\mathrm{w.o.t.}}$ according to the Bicommutant Theorem. Plainly, $\phi^\dag$ is a~bounded linear operator. We shall prove that it is $\eoz{\e}$

Fix two elements
$$
x+\alpha\cdot 1_{A^\dag},\, y+\beta\cdot 1_{A^\dag}\in \Aa^\dag_+\quad\mbox{such that}\quad x+\alpha\cdot 1_{\Aa^\dag}\perp y+\beta\cdot 1_{\Aa^\dag}.
$$
Since orthogonality passes to quotient algebras, we have $\alpha=0$ or $\beta=0$. With no loss of generality assume that $\beta=0$ and hence $y\in \Aa_+$. Notice that $x\in \Aa_\sa$ and $\alpha\geq 0$, as positivity is preserved by quotient algebras as well. For any fixed $\lambda\in\Lambda$ we have
\begin{equation*}
\begin{split}
x+\alpha\cdot 1_{\Aa^\dag}=(x &+\alpha\cdot 1_{\Aa^\dag})^{1/2}(1_{\Aa^\dag}-u_\lambda)(x+\alpha\cdot 1_{\Aa^\dag})^{1/2}\\
& +(x+\alpha\cdot 1_{\Aa^\dag})^{1/2}u_\lambda (x+\alpha\cdot 1_{\Aa^\dag})^{1/2}.
\end{split}
\end{equation*}
Observe that the latter summand belongs to $\Aa$ and is dominated by $x+\alpha\cdot 1_{\Aa^\dag}$, therefore
$$
y\perp (x+\alpha\cdot 1_{\Aa^\dag})^{1/2}u_\lambda (x+\alpha\cdot 1_{\Aa^\dag})^{1/2}.
$$
Since $\phi$ is $\eoz{\e}$, we have
\begin{equation}\label{phiplus}
\begin{split}
\bigl\|\phi^\dag(y)\phi^\dag\bigl((x &+\alpha\cdot 1_{\Aa^\dag})^{1/2}u_\lambda (x+\alpha\cdot 1_{\Aa^\dag})^{1/2}\bigr)\bigr\|\\
&\leq\e\bigl\|y\bigr\|\bigl\|(x+\alpha\cdot 1_{\Aa^\dag})^{1/2}u_\lambda (x+\alpha\cdot 1_{\Aa^\dag})^{1/2}\bigr\|\\
&\leq\e\bigl\|y\bigr\|\bigl\|x+\alpha\cdot 1_{\Aa^\dag}\bigr\|.
\end{split}
\end{equation}
Set
$$
z=\phi^\dag(y)\phi^\dag(x+\alpha\cdot 1_{\Aa^\dag});
$$
note that $z=z_{1,\lambda}+z_{2,\lambda}$, where
$$
z_{1,\lambda}=\phi^\dag\bigl(y\bigr)\phi^\dag\big((x+\alpha\cdot 1_{\Aa^\dag})^{1/2}u_\lambda (x+\alpha\cdot 1_{\Aa^\dag})^{1/2}\big)
$$
and
$$
z_{2,\lambda}=\phi^\dag\big(y\big)\phi^\dag\big((x+\alpha\cdot 1_{\Aa^\dag})^{1/2}(1_{\Aa^\dag}-u_\lambda)(x+\alpha\cdot 1_{\Aa^\dag})^{1/2}\big).
$$
Under this notation we have
\begin{equation}\label{z12}
\n{z_{1,\lambda}}\leq\e\n{y}\n{x+\alpha\cdot 1_{\Aa^\dag}}\quad\mbox{ and }\quad z_{2,\lambda}\xrightarrow[\lambda]{\phantom{x}\mathrm{w.o.t.}\phantom{x}}0.
\end{equation}
In fact, the former statement is just a~rewriting of \eqref{phiplus}. For the~latter one observe that since $(u_\lambda)_{\lambda\in\Lambda}$ is approximately central in $\Aa$, we have
\begin{equation*}
\begin{split}
\mathrm{w.o.t.}\lim_{\lambda}\phi^\dag\big( &(x+\alpha\cdot 1_{\Aa^\dag})^{1/2}(1_{\Aa^\dag}-u_\lambda)(x+\alpha\cdot 1_{\Aa^\dag})^{1/2}\big)\\
&=\mathrm{w.o.t.}\lim_{\lambda}\phi^\dag\big((1_{\Aa^\dag}-u_\lambda)(x+\alpha\cdot 1_{\Aa^\dag})\big)\\
&=\mathrm{w.o.t.}\lim_{\lambda}\phi^\dag\big(\alpha(1_{\Aa^\dag}-u_\lambda)\big)\\
&=\alpha\big(z_0-\mathrm{w.o.t.}\lim_\lambda\phi^\dag(u_\lambda)\big)=0.
\end{split}
\end{equation*}
Now, the desired inequality $\n{z}\leq\e\n{y}\n{x+\alpha\cdot 1_{\Aa^\dag}}$ follows easily. Indeed, otherwise we could pick $\xi,\eta\in\Hh$ with $\n{\xi}=\n{\eta}=1$ and such that
$$
\langle z(\xi),\eta\rangle=\langle z_{1,\lambda}(\xi),\eta\rangle+\langle z_{2,\lambda}(\xi),\eta\rangle>\n{y}\n{x+\alpha\cdot 1_{\Aa^\dag}}.
$$
Passing to limit over $\lambda\in\Lambda$ we obtain a~contradiction with \eqref{z12}.

We shall now prove that $\phi^\dag$ is $\esa{6\delta}$ provided that $\phi$ is $\esa{\delta}$. To this end, notice that since the~involution is weakly continuous, we have
$$
z_0-z_0^\ast=\mathrm{w.o.t.}\lim_\lambda\pi(\phi(u_\lambda)-\phi(u_\lambda)^\ast),
$$
thus $\n{z_0-z_0^\ast}\leq \delta$, in view of the fact that $\n{\pi}\leq 1$ and $\n{u_\lambda}\leq 1$ for each $\lambda\in\Lambda$. Now, for any $x\in \Aa_\sa$ and $\alpha\in\C$ we have
$$
\phi^\dag((x+\alpha\cdot 1_{\Aa^\dag})^\ast)-(\phi^\dag(x+\alpha\cdot 1_{\Aa^\dag}))^\ast=\pi\circ\phi(x)-(\pi\circ\phi(x))^\ast+\overline{\alpha}z_0-\overline{\alpha}z_0^\ast,
$$
hence the norm of the left-hand side is at most 
$$
\n{\phi(x)-\phi(x)^\ast}+\abs{\alpha}\n{z_0-z_0^\ast}\leq \delta(\n{x}+\alpha)\leq 3\delta\n{x+\alpha\cdot 1_{\Aa^\dag}},
$$
where the last estimate follows from the above mentioned Gaur--Kov\'a\v{r}\'{i}k inequality.
We have thus shown that 
$$
\n{\phi^\dag(z)-\phi^\dag(z)^\ast}\leq 3\delta\n{z}\quad\mbox{for every }z\in (\Aa^\dag)_\sa.
$$
From this it immediately follows that $\phi^\dag$ is $\esa{6\delta}$ by splitting any element of $\Aa^\dag$ into its real and imaginary parts.

The assertion that $\phi^\dag$ is completely positive whenever $\phi$ is can be proved by appealing to the Stinespring's theorem in the same way as in the~proof of \cite[Prop.~3.2]{WZ1}. Note that in this case the weak limit defining $\phi^\dag(1_{\Aa^\dag})$ can be replaced by the~strong limit after picking an~increasing net $(u_\lambda)_{\lambda\in\Lambda}$, since then the net $(\phi(u_\lambda))_{\lambda\in\Lambda}$ is bounded and monotone increasing.
\end{proof}

\begin{proposition}\label{P_almostJordan}
Let $\Aa$ and $\Bb$ be \cs-algebras and assume $\Aa$ is unital. If $\phi\in\LL(\Aa,\Bb)$ is an~$\eoz{\e}$ operator with $h\coloneqq \phi(1_\Aa)$, then
$$
\n{\phi(x)^2-h\phi(x^2)}\leq 108\e\n{x}^2\quad\mbox{for every }x\in \Aa.
$$
\end{proposition}
\begin{proof}
First, we shall prove that
\begin{equation}\label{partial_Jordan}
\n{\phi(x)^2-h\phi(x^2)}\leq 8\e\n{x}^2\quad\mbox{for every }x\in \Aa_+.
\end{equation}
By homogeneity, it is enough to consider any $x\in \Aa_\sa$ such that $0\leq x\leq 1_\Aa$. In view of the Gelfand--Naimark theorem, we have an isomorphism \cs$(x,1_\Aa)\cong C(\sigma(x))$ between the \cs-subalgebra of $\Aa$ generated by $\{x,1_\Aa\}$ and the algebra of complex-valued continuous functions on the spectrum $\sigma(x)\subseteq [0,1]$, where $x$ corresponds to the~identity function $\mathrm{id}_{\sigma(x)}$. By this identification we can regard $\phi$ as an~operator defined on $C(\sigma(x))$. Its second adjoint $\phi^{\ast\ast}$ is then defined on $C(\sigma(x))^{\ast\ast}$ which contains the space of all bounded Borel functions on $\sigma(x)$.

For any $n\in\N$, we consider a partition of $\sigma(x)$ given by
$$
X_{0,n}=\Bigl[0,\frac{1}{n}\Bigr]\cap\sigma(x),\,\,\, X_{1,n}=\Bigl(\frac{1}{n},\frac{2}{n}\Bigr]\cap\sigma(x),\,\ldots,\,\,\,\, X_{n-1,n}=\Bigl(\frac{n-1}{n},1\Bigr]\cap\sigma(x)
$$
and we pick arbitrary points $x_{k,n}\in X_{k,n}$ for $0\leq k<n$ (if some $X_{k,n}=\varnothing$, we ignore the symbol $x_{k,n}$ in all computations below). Consider any $f\in C(\sigma(x))$ regarded canonically as an~element of $C(\sigma(x))^{\ast\ast}$. Define
$$
f_n=\sum_{k=0}^{n-1}f(x_{k,n})\ind_{X_{k,n}}\quad\mbox{for }n\in\N
$$
and observe that $f_n\xrightarrow[]{\,w^\ast\,}f$. Therefore,
\begin{equation}\label{phi(f)}
\phi(f)=\phi^{\ast\ast}(f)=\lim_{n\to\infty}\phi^{\ast\ast}(f_n)=\lim_{n\to\infty}\sum_{k=0}^{n-1}f(x_{k,n})\phi^{\ast\ast}(\ind_{X_{k,n}})
\end{equation}
and, consequently,
\begin{equation}\label{difference}
\begin{split}
\phi(f)^2-h\phi(f^2)=\lim_{n\to\infty}\Biggl\{ &\sum_{0\leq j\not=k<n} f(x_{j,n})f(x_{k,n})\phi^{\ast\ast}(\ind_{X_{j,n}})\phi^{\ast\ast}(\ind_{X_{k,n}})\\
& +\sum_{0\leq j<n} f(x_{j,n})^2\Bigl[\phi^{\ast\ast}(\ind_{X_{j,n}})^2-\phi^{\ast\ast}(\ind_{\sigma(x)})\phi^{\ast\ast}(\ind_{X_{j,n}})\Bigr]\Biggr\}.
\end{split}
\end{equation}
We are going to estimate the norms of both the~above sums separately. To this end, we start with the~following observation: Given any two sets $A$ and $B$ of the form
$$
A=\vert a,b]\cap\sigma(x),\,\,\,\, B=(c,d]\cap\sigma(x),\quad\mbox{where } 0\leq a<b\leq c<d\leq 1,
$$
we have $\n{\phi^{\ast\ast}(\ind_A)\phi^{\ast\ast}(\ind_B)}\leq\e$. Indeed, let us define, for each $n\in\N$, piecewise linear maps $\widetilde{e}_n, \widetilde{g}_n\colon [0,1]\to [0,1]$ by
$$
\widetilde{e}_n(t)=\left\{\begin{array}{rl}
0 & \mbox{if }\,0\leq t\leq a+\frac{b-a}{2^{n+1}}\,\mbox{ or }\,b+2^{-n}\leq t\leq 1\\
1 & \mbox{if }\,a+\frac{b-a}{2^n}\leq t\leq b\\
& \hspace*{-6mm}\mbox{continuous and linear elsewhere},
\end{array}\right.
$$
and $\widetilde{g}_n$ by a~similar formula, where the~endpoints $a$ and $b$ are replaced by $c$ and $d$, respectively. Set also $e_n=\widetilde{e}_n\vert_{\sigma(x)}$ and $g_n=\widetilde{g}_n\vert_{\sigma(x)}$. By Lebesgue's theorem, we have $e_n\xrightarrow[]{\,w\ast\,}\ind_A$ and $g_n\xrightarrow[]{\,w\ast\,}\ind_B$. Observe also that $e_ng_m=0$ for every $m\in\N$ and $n$ sufficiently large. By the assumption, for all such pairs $(m,n)$ we have $\n{\phi(e_n)\phi(g_m)}\leq\e$. Fixing $m\in\N$ and passing with $n$ to infinity we obtain $\n{\phi^{\ast\ast}(\ind_A)\phi(g_m)}\leq\e$, as multiplication is separately continuous with respect to the weak$^\ast$ topology on $C(\sigma(x))^{\ast\ast}$ and $\phi^{\ast\ast}$ is weak$^\ast$-to-weak$^\ast$ continuous. Next, passing with $m$ to infinity we obtain the announced inequality. 

Note that a similar reasoning, with suitably modified $e_n$'s and $g_n$'s, applies in the case where $A$ and $B$ are disjoint finite unions of intervals intersected with $\sigma(x)$. Moreover, it is easily seen that the~argument goes through if we multiply each term of the form $\phi^{\ast\ast}(\ind_{I\cap\sigma(x)})$, where $I$ is an~interval, by any weight of modulus at most one. Hence, for every function $f\in C(\sigma(x))$ with $\n{f}_\infty\leq 1$ and for any disjoint sets $M,N\subset\{0,1,\dots,n-1\}$ we have
$$
\Biggl\|\Bigl(\sum_{j\in M} f(x_{j,n})\phi^{\ast\ast}(\ind_{X_{j,n}})\Bigr)\Bigl(\sum_{k\in N} f(x_{k,n})\phi^{\ast\ast}(\ind_{X_{k,n}})\Bigr)\Biggr\|\leq\e.
$$
In what follows, we still assume that $f\in C(\sigma(x))$ and $\n{f}_\infty\leq 1$.

Denote by $\Pi$ the collection of all nontrivial ordered partitions of $\{0,1,\ldots,n-1\}$, that is, $\Pi$ consists of all pairs $(M,N)$ with $M\not=\varnothing\not=N$, $M\cap N=\varnothing$ and $M\cup N=\{0,1,\ldots,n-1\}$. Obviously, we have $\abs{\Pi}=2^n-2$ and hence
$$
\sum_{\{M,N\}\in\Pi}\,\Biggl\|\sum_{j\in M,\, k\in N}f(x_{j,n})f(x_{k,n})\phi^{\ast\ast}(\ind_{X_{j,n}})\phi^{\ast\ast}(\ind_{X_{k,n}})\Biggr\|\leq (2^{n}-2)\e.
$$
Notice that for any fixed integers $0\leq j\not=k<n$, the number of partitions from $(M,N)\in\Pi$ that separate $j$ and $k$ and satisfy $j\in M$ equals $2^{n-2}$. Therefore, by the triangle inequality, we obtain
\begin{equation}\label{a_estimate}
\Biggl\|\sum_{0\leq j\not=k<n} f(x_{j,n})f(x_{k,n})\phi^{\ast\ast}(\ind_{X_{j,n}})\phi^{\ast\ast}(\ind_{X_{k,n}})\Biggr\|\leq \frac{2^{n}-2}{2^{n-2}}\e\xrightarrow[n\to \infty]{} 4\e.
\end{equation}

In order to estimate the norm of the second sum in formula \eqref{difference} we consider only even integers $n$. Each summand can be written in the form
\begin{equation*}
\begin{split}
s_j &\coloneqq f(x_{j,n})^2\Bigl[\phi^{\ast\ast}(\ind_{X_{j,n}})^2-\phi^{\ast\ast}(\ind_{\sigma(x)})\phi^{\ast\ast}(\ind_{X_{j,n}})\Bigr]\\
&\,=-f(x_{j,n})^2\phi^{\ast\ast}(\ind_{X_{j,n}})\phi^{\ast\ast}(\ind_{\sigma(x)\setminus X_{j,n}})\\
&\,=-f(x_{j,n})^2\phi^{\ast\ast}(\ind_{X_{j,n}})\sum_{k\not=j}\phi^{\ast\ast}(\ind_{X_{k,n}}).
\end{split}
\end{equation*}
Let $\Pi^\prime$ be the collection of all ordered partitions $(M,N)$ of $\{0,1,\ldots,n-1\}$ such that $\abs{M}=\abs{N}=n/2$. Note that $\abs{\Pi^\prime}=\binom{n}{n/2}$ and that for every $(M,N)\in\Pi^\prime$ we have 
$$
\Biggl\|\Bigl(\sum_{j\in M} f(x_{j,n})^2\phi^{\ast\ast}(\ind_{X_{j,n}})\Bigr)\Bigl(\sum_{k\in N}\phi^{\ast\ast}(\ind_{X_{k,n}})\Bigr)\Biggr\|\leq\e.
$$
Summing up all these inequalities we obtain
\begin{equation}\label{summing_up}
\Biggl\|\sum_{(M,N)\in\Pi^\prime} \sum_{j\in M}f(x_{j,n})^2\phi^{\ast\ast}(\ind_{X_{j,n}})\sum_{k\in N}\phi^{\ast\ast}(\ind_{X_{k,n}})\Biggr\|\leq \binom{n}{n/2}\e.
\end{equation}
Notice that for any fixed $j\in\{0,1,\ldots,n-1\}$ we have
$$
\bigl|\{(M,N)\in\Pi^\prime\colon j\in M\}\bigr|=\binom{n-1}{n/2}
$$
and every partition as above gives rise to an~expression $f(x_{j,n})^2\phi^{\ast\ast}(\ind_{X_{j,n}})\sum_{k\in N}\phi^{\ast\ast}(\ind_{X_{k,n}})$, where the last sum has $n/2$ summands. By symmetry, each $s_j$ is realized under the norm sign in \eqref{summing_up} exactly $\frac{n}{2}\binom{n-1}{n/2}\frac{1}{n-1}$ times. It is indeed an~integer which can be written in the form
\begin{equation*}
\begin{split}
\frac{1}{n-1}\cdot\frac{n}{2}\cdot\frac{(n-1)(n-2)\cdot\ldots\cdot (n-\frac{n}{2})}{1\cdot 2\cdot\ldots\cdot \frac{n}{2}} &=\frac{(n-2)(n-3)\cdot\ldots\cdot (n-\frac{n}{2})}{1\cdot 2\cdot\ldots\cdot (\frac{n}{2}-1)}\\
&=\binom{n-2}{n/2-1}=\frac{n}{4(n-1)}\binom{n}{n/2}.
\end{split}
\end{equation*}
Therefore, the sum under the norm sign in \eqref{summing_up} equals $\frac{n}{4(n-1)}\binom{n}{n/2}\sum_{0\leq j<n} s_j$ and hence
\begin{equation}\label{b_estimate}
\Biggl\|\sum_{0\leq j<n} f(x_{j,n})^2\Bigl[\phi^{\ast\ast}(\ind_{X_{j,n}})^2-\phi^{\ast\ast}(\ind_{\sigma(x)})\phi^{\ast\ast}(\ind_{X_{j,n}})\Bigr]\Biggr\|\leq\frac{4(n-1)}{n}\e\xrightarrow[n\to\infty]{}4\e.
\end{equation}
Combining \eqref{a_estimate} and \eqref{b_estimate} with formula \eqref{difference} applied to the~function $f=\mathrm{id}_{\sigma(x)}$, we obtain the~announced inequality \eqref{partial_Jordan}.

The result now follows by splitting an arbitrary $x\in \Aa$ into its real and imaginary parts, and applying the Jordan decomposition to each of them. Indeed, observe that if $y\in \Aa_\sa$, $y=y_1-y_2$, where $y_1,y_2\in \Aa_+$ and $y_1y_2=0$, then 
$$
\phi(y^2)=\phi(y_1^2+y_2^2)=\phi(y_1^2)+\phi(y_2^2)
$$
and
$$
\phi(y)^2=\phi(y_1)^2+\phi(y_2)^2-\phi(y_1)\phi(y_2)-\phi(y_2)\phi(y_1).
$$ 
Hence, making use of \eqref{partial_Jordan}, we get
\begin{equation*}
\begin{split}
\n{\phi(y)^2-h\phi(y^2)} &\leq \n{\phi(y_1)\phi(y_2)}+\n{\phi(y_2)\phi(y_1)}+8\e(\n{y_1}^2+\n{y_2}^2)\\
& \leq 2\e\n{y_1}\n{y_2}+8\e(\n{y_1}^2+\n{y_2}^2)\leq 18\e\n{y}^2.
\end{split}
\end{equation*}
Now, if $x=y+\ii z$ with $y,z\in \Aa_\sa$, then we have
\begin{equation*}
\begin{split}
\phi(x^2) & =\phi(y^2)-\phi(z^2)+\ii\phi(yz+zy)\\
&=\phi(y^2)-\phi(z^2)+\frac{1}{2}\ii\bigl(\phi((y+z)^2)-\phi((y-z)^2)\bigr)
\end{split}
\end{equation*}
and
\begin{equation*}
\begin{split}
\phi(x)^2 & =\phi(y)^2-\phi(z)^2+\ii(\phi(y)\phi(z)+\phi(z)\phi(y))\\
&=\phi(y)^2-\phi(z)^2+\frac{1}{2}\ii\bigl(\phi(y+z)^2-\phi(y-z)^2\bigr).
\end{split}
\end{equation*}
Therefore,
\begin{equation*}
\n{\phi(x)^2-h\phi(x^2)}\leq 18\e(\n{y}^2+\n{z}^2)+9\e(\n{y+z}^2+\n{y-z}^2)\leq 108\e\n{x}^2.\qedhere
\end{equation*}
\end{proof}

\section{The second adjoint on bounded Borel functions}
\noindent
For future use, we shall isolate a~part of the~proof of Proposition~\ref{P_almostJordan} (the~one about disjointness preserving properties of the~second adjoint operator) and give it a~somewhat stronger form. Before doing it, note that if $\Aa_0\subseteq\Aa$ is a~commutative \cs-subalgebra of $\Aa$, then $\Aa^\dast$ contains a~\cs-subalgebra isomorphic to the~algebra $B(\sigma(\Aa_0))$ of bounded Borel functions on the~spectrum of $\Aa_0$ (see \cite[\S III.5.13]{blackadar}). Thus, for any $\phi\in\LL(\Aa,\Bb)$ it makes sense to speak about the~restriction of $\phi^\dast$ to $B(\sigma(\Aa_0))$.
\begin{proposition}\label{Borel_lemma}
Let $\Aa$ and $\Bb$ be \cs-algebras and $\phi\in\LL(\Aa,\Bb)$ be an~$\eoz{\e}$ map. Then, for every commutative separable \cs-subalgebra $\Aa_0$ of $\Aa$, the~operator $\phi^\dast\!\restriction_{B(\sigma(\Aa_0))}$ is also ~$\eoz{\e}$
\end{proposition}

Before proving this assertion let us collect essential tools from descriptive set theory. As usual, we denote by $\mathcal{N}=\N^\N$ the~Baire space of all countably infinite sequences of natural numbers. For any metric space $X$ we write $\mathfrak{B}(X)$ for the $\sigma$-algebra of Borel subsets of $X$, and we use the~standard notation: 
$$
\begin{array}{l}
\Sigma_1^0(X)=\{U\subseteq X\colon U\mbox{ is open}\},\\
\Pi_\xi^0(X)=\{X\setminus D\colon D\in \Sigma_\xi^0(X)\},\\
\Sigma_\xi^0(X)=\Bigl\{\bigcup_{n\in\N} D_n\colon D_n\in\Pi_{\xi_n}^0(X),\,\, \xi_n<\xi,\,\, n\in\N\Bigr\}
\end{array}
$$
\renewcommand{\arraystretch}{1}for any ordinal $1<\xi<\omega_1$. Plainly, $\mathfrak{B}(X)=\ccup_{\xi<\omega_1}\Sigma_\xi^0(X)=\ccup_{\xi<\omega_1}\Pi_\xi^0(X)$ (see, e.g., \cite[\S II.11]{kechris}). We also use standard notation for function spaces: $C_0(X)$ for continuous and vanishing at infinity functions, $C_b(X)$ for continuous bounded functions, and $B(X)$ for Borel bounded functions (all complex-valued and defined on $X$).

\renewcommand{\arraystretch}{1.6}
Recall that for ordinals $1\leq\xi <\omega_1$ the Baire classes $B_\xi(X)$ are defined as follows: A~function $f\colon X\to\C$ is {\it of Baire class} $1$ if $f^{-1}(U)$ is an~$F_\sigma$-set for every open set $U\subseteq\C$. As we consider complex-valued functions, it is equivalent to saying that $f$ is the~pointwise limit of a~sequence of continuous functions ({\it cf.} \cite[Thm.~24.10]{kechris}). Next, for $1<\xi<\omega_1$, we say that $f$ is {\it of Baire class} $\xi$ provided that $f$ is the~pointwise limit of a~sequence of functions $f_n\colon X\to\C$, where each $f_n$ is of Baire class $\xi_n$ with some $\xi_n<\xi$. In a~similar fashion we define classes $B_\xi(X,Y)$ consisting of Baire class $\xi$ functions from $X$ to $Y$, where $Y$ is any separable metric space. (One difference is that, in general, Baire class $1$ functions may not be pointwise limits of sequences of continuous functions.) According to the~theorem of Lebesgue, Hausdorff and Banach (\cite[Thm.~24.3]{kechris}), the~union $\ccup_{1\leq\xi<\omega_1} B_\xi(X,Y)$ is the~whole class of Borel functions mapping $X$ into $Y$.

\begin{theorem}[{Lusin--Souslin; see \cite[Thm.~13.7]{kechris}}]\label{LS_theorem}
Let $X$ be a Polish space and $A\subseteq X$ a~Borel set. Then, there exists a~closed set $F\subseteq\mathcal{N}$ and a~continuous bijection $\Phi\colon F\to A$.
\end{theorem}
We will also use the following classical result which plays a~key role in the proof of the above quoted theorem of Lusin and Souslin.
\begin{theorem}[{\cite[Thm.~13.1]{kechris}}]\label{Polish_topology}
Let $(X,\mathcal{T})$ be a Polish space and $A\subseteq X$ a~Borel set. Then, there exists a~Polish topology $\mathcal{T}_A\supseteq\mathcal{T}$ such that $\mathfrak{B}(\mathcal{T}_A)=\mathfrak{B}(\mathcal{T})$ and $A$ is clopen with respect to $\mathcal{T}_A$.
\end{theorem}

\begin{lemma}\label{first_class_L}
Let $X$ be a metric space. Every nonnegative bounded function $f\in B_1(X)$ is the~pointwise limit of a~sequence $(f_n)_{n=1}^\infty$ of simple functions $f_n\colon X\to [0,\infty)$ such that for every $n\in\N$ we have $f_n^{-1}(0,\infty)\subseteq f^{-1}(0,\infty)$, $\norm{f_n}_\infty\leq\norm{f}_\infty$ and $f_n^{-1}(\{t\})\in\Sigma_2^0(X)$ for any $t>0$.
\end{lemma}
\begin{proof}
Obviously, for each $n\in\N$ we can pick a~sequence $0=t_{n,0}<t_{n,1}<t_{n,2}<\ldots<t_{n,k_n}$ such that:
\begin{enumerate}[\tiny$\bullet$, itemindent=-16pt, itemsep=3pt]
\item $t_{n,k_n-1}<\norm{f}_\infty<t_{n,k_n}$ for each $n\in\N$;
\item $t_{n,i}-t_{n,i-1}<1/n$ for all $n\in\N$, $1\leq i\leq k_n$;
\item $\{t_{m,i}\colon 1\leq m<n,\, 1\leq i\leq k_m\}\cap\{t_{n,i}\colon 1\leq i\leq k_n\}=\varnothing$ for each $n\in\N$.
\end{enumerate}
Set $D_{n,i}=\{x\in X\colon t_{n,i-1}<f(x)<t_{n,i}\}$ for all $n\in\N$, $1\leq i\leq k_n$, and define a~simple function $f_n\colon X\to [0,\infty)$ by
$$
f_n=\sum_{i=1}^{k_n} t_{n,i-1}\ind_{D_{n,i}}.
$$
Since $f$ is of Baire class $1$, all the sets $D_{n,i}$ are in $\Sigma_2^0(X)$. Also, $f_n$ vanishes whenever $f$ does and obviously we have $\norm{f_n}_\infty\leq\norm{f}_\infty$ for every $n\in\N$. 

Finally, note that $f_n\xrightarrow[]{\quad}f$ pointwise. Indeed, fix any $x\in X$ with $f(x)>0$ and observe that $f(x)=t_{m,i}$ for at most one pair $(m,i)$ with $m\in\N$, $1\leq i\leq k_m$. Then, for every $n>m$ (or every $n\in\N$ if that $m$ does not exist) we have $\abs{f_n(x)-f(x)}<\frac{1}{n}$.
\end{proof}

\begin{proof}[Proof of Proposition \ref{Borel_lemma}]
Note that since $\Aa_0\cong C_0(\sigma(\Aa_0))$ is separable, the~spectrum $\sigma(\Aa_0)$ is metrizable (and, as always, locally compact). Moreover, separability of $\Aa_0$ implies that $\Aa_0$ is $\sigma$-unital, i.e. has a~countable approximate unit, and hence $\sigma(\Aa_0)$ is also $\sigma$-compact. This implies that $\sigma(\Aa_0)$ is a~Polish space (see, e.g., \cite[Thm.~5.3]{kechris}). Let $\varrho$ be a~complete metric compatible with its topology. 

Denote by $\mathcal{S}$ the collection of all nonnegative simple functions in $B(\sigma(\Aa_0))$. Each $f\in\mathcal{S}$ can be written uniquely as $f=\sum_{i=1}^k\alpha_i\ind_{D_i}$ with $k\geq 0$, $\alpha_i>0$ and mutually disjoint nonempty sets $D_i\in\mathfrak{B}(\sigma(\Aa_0))$ ($1\leq i\leq k$). We then denote $\mathsf{D}(f)=\{D_1,\ldots,D_k\}$. Define $\mathfrak{D}$ to be the~largest class contained in $\mathfrak{B}(\sigma(\Aa_0))$ with the~following property: For all $f_1, f_2\in\mathcal{S}$ satisfying $f_1f_2=0$ and $\mathsf{D}(f_1)\cup\mathsf{D}(f_2)\subset\mathfrak{D}$,  we have
\begin{equation}\label{phi_dast}
\n{\phi^\dast(f_1)\phi^\dast(f_2)}\leq\e\n{f_1}_\infty\n{f_2}_\infty.
\end{equation}
We are to prove that $\mathfrak{D}$ is the~whole of $\mathfrak{B}(\sigma(\Aa_0))$.

\vspace*{2mm}\noindent
{\it Claim 1. }$\Sigma_2^0(\sigma(\Aa_0))\subseteq\mathfrak{D}$. 

\vspace*{1mm}
We first prove that $\mathfrak{D}$ contains all the closed sets. 

Fix any $f_1,f_2\in\mathcal{S}$ with $\mathsf{D}(f_1)\cup\mathsf{D}(f_2)\subset\Pi_1^0(\sigma(\Aa_0))$ and such that $f_1f_2=0$ which means nothing but $\ccup\,\mathsf{D}(f_1)\cap\ccup\,\mathsf{D}(f_2)=\varnothing$. For now, assume additionally that both $f_1$ and $f_2$ are compactly supported. For $i=1,2$ and every $n\in\N$, define closed sets
$$
F_{n,i}\coloneqq\Big\{\lambda\in \sigma(\Aa_0)\colon \dist_\varrho\Big(\lambda,\,\bigcup\mathsf{D}(f_i)\Big)\geq\frac{1}{n}\Bigr\}
$$
and the standard (continuous) ``cut-off" functions
$$
f_{n,i}(\lambda)\coloneqq\displaystyle{\frac{\dist_\rho(\lambda, F_{n,i})}{\dist_\rho(\lambda, F_{n,i})+\dist_\rho\big(\lambda, \bigcup\,\mathsf{D}(f_i)\big)}}.
$$
Notice that $f_{n,i}$ vanishes outside the set $U_{n,i}\coloneqq\{\lambda\colon\dist_\varrho(\lambda,\ccup\,\mathsf{D}(f_i))<1/n\}$. Since $\sigma(\Aa_0)$ is locally compact and $\ccup\,\mathsf{D}(f_i)$ has compact closure, the~set $U_{n,i}$ has compact closure for $n$ large enough. Indeed, there are finitely many points $\lambda_1,\ldots,\lambda_m\in\ccup\,\mathsf{D}(f_i)$ and positive numbers $r_1,\ldots, r_m$ such that the open balls $B(\lambda_j,r_j)$ cover the~whole of $\ccup\,\mathsf{D}(f_i)$ and $B(\lambda_j,2r_j)$ have compact closures ($1\leq j\leq m$). Hence, if $\dist_\varrho(\lambda, \ccup\,\mathsf{D}(f_i))<\min_{j} r_j$, then $\varrho(\lambda,\lambda_j)<2r_j$ for some $1\leq j\leq m$ which means that for each $n\geq (\min_j r_j)^{-1}$ the~set $U_{n,i}$ is contained in the~relatively compact set $\ccup_j B(\lambda_j,r_j)$.

Since the closed sets $\ccup\,\mathsf{D}(f_1)$ and $\ccup\,\mathsf{D}(f_2)$ are at positive distance, we have $U_{n,1}\cap U_{n,2}=\varnothing$ for sufficiently large $n$. Hence, we may pick $n_0\in\N$ so that for each $n\geq n_0$ we have $f_{n,1}f_{n,2}=0$ and $f_{n,1}$, $f_{n,2}$ have compact supports.

For $i=1,2$ write $f_i=\sum_{D\in\mathsf{D}(f_i)}\alpha_i(D)\ind_{D}$. By appealing to Urysohn's lemma, we pick continuous functions $g_1,g_2\geq 0$ on $\sigma(\Aa_0)$ such that $g_i\!\!\restriction_{D}=\alpha_i(D)$ and $\n{g_i}_\infty=\n{f_i}_\infty$ for $i=1,2$ and $D\in\mathsf{D}(f_i)$. Then $(g_if_{n,i})_{n\geq n_0}$ is a~sequence of compactly supported nonnegative continuous functions converging pointwise, and hence weak$^\ast$, to $f_i$ ($i=1,2$). Since $\phi$ is $\eoz{\e}$, we have 
$$
\n{\phi(g_1f_{n,1})\phi(g_2f_{n,2})}\leq \e\norm{g_1f_{n,1}}_\infty\norm{g_2f_{n,2}}_\infty\!=\e\norm{f_1}_\infty\norm{f_2}_\infty\quad\mbox{for each }n\geq n_0.
$$
Using the fact that multiplication is separately weak$^\ast$ continuous and $\phi^\dast$ is weak$^\ast$-to-weak$^\ast$ continuous (as in the~proof of Proposition~\ref{P_almostJordan}) we obtain inequality \eqref{phi_dast}. 

If we drop the assumption that $f_1$ and $f_2$ are compactly supported, using $\sigma$-compactness of $\sigma(\Aa_0)$ we pick sequences $(f_{n,i})_{n=1}^\infty\subset\mathcal{S}$ ($i=1,2$) of compactly supported functions such that:
\begin{enumerate}[\tiny$\bullet$, itemindent=-16pt, itemsep=3pt]
\item $f_{n,i}\xrightarrow[]{\quad}f_i$ pointwise for $i=1,2$;
\item $\mathsf{D}(f_{n,i})\subset \Pi_1^0(\sigma(\Aa_0))$ for all $n\in\N$, $i=1,2$;
\item $f_{n,1}f_{n,2}=0$ for every $n\in\N$.
\end{enumerate}
Inequality \eqref{phi_dast} then follows from the previous part by passing to limit as above. This concludes the~argument for $\Pi_1^0(X)\subseteq\mathfrak{D}$.

The next step is to observe that $\mathfrak{D}$ is closed under countable unions. Indeed, let $f_1,f_2\in\mathcal{S}$ satisfy $f_1f_2=0$, $\mathsf{D}(f_1)=\{D_1,\ldots,D_k\}$ and $\mathsf{D}(f_2)=\{E_1,\ldots,E_l\}$, where for all $1\leq i\leq k$, $1\leq j\leq l$ we have $D_i=\ccup_{m=1}^\infty D_{i,m}$ and $E_j=\ccup_{m=1}^\infty E_{j,m}$ with all the~$D_{i,m}$, $E_{j,m}$ belonging to $\mathfrak{D}$. For $m\in\N$, define 
$$
f_{1,m}=\Big(\sum_{r=1}^k\sum_{s=1}^m \ind_{D_{r,s}}\Big)\!\cdot\! f_1\quad\mbox{ and }\quad f_{2,m}=\Big(\sum_{r=1}^l\sum_{s=1}^m \ind_{E_{r,s}}\Big)\!\cdot\! f_2,
$$
and notice that $f_{i,m}\xrightarrow[]{\quad} f_i$ pointwise for $i=1,2$, whereas by the~definition of $\mathfrak{D}$, we have
$$
\norm{\phi^\dast(f_{1,m})\phi^\dast(f_{2,m})}\leq\e\norm{f_{1,m}}_\infty\norm{f_{2,m}}_\infty\leq\e\norm{f_1}_\infty\norm{f_2}_\infty\quad\mbox{for each }m\in\N.
$$ 
Again, by passing to limit we obtain inequality \eqref{phi_dast}. {\it Claim~1} has thus been established.

\vspace*{1mm}
Now, suppose $\mathcal{F}=\{A_1,\ldots,A_m\}$ is an arbitrary finite family of mutually disjoint sets in $\mathfrak{B}(\sigma(\Aa_0))$. According to Theorem~\ref{Polish_topology}, there is a~sequence 
$$
\mathcal{T}_{A_1}\subseteq\mathcal{T}_{A_1A_2}\subseteq\ldots\subseteq\mathcal{T}_{A_1A_2\ldots A_m}
$$
of Polish topologies on $\sigma(\Aa_0)$, all finer than the~original one, for which $\mathfrak{B}(\sigma(\Aa_0))=\mathfrak{B}(\mathcal{T}_{A_1A_2\ldots A_m})$ and such that each of $A_1,\ldots,A_m$ is clopen in $\mathcal{T}_{A_1A_2\ldots A_m}$. Theorem~\ref{LS_theorem} produces a~closed set $F\subseteq\mathcal{N}$ and a~continuous bijection $\Phi\colon F\to (\sigma(\Aa_0),\mathcal{T}_{A_1A_2\ldots A_m})$ (which is, of course, continuous also with respect to $\varrho$). Plainly, $\{\Phi^{-1}(A_i)\colon 1\leq i\leq m\}$ is then a~collection of mutually disjoint closed subsets of $\mathcal{N}$.

\vspace*{2mm}\noindent
{\it Claim 2. }The operator $\varphi\in\LL(C_b(F),\Bb^\dast)$ defined by $\varphi(f)=\phi^\dast(f\circ\Phi^{-1})$ is $\eoz{\e}$

\vspace*{1mm}
First of all, note that $\varphi$ is well-defined because $\Phi^{-1}\colon\sigma(\Aa_0)\to F$ is a~Borel map (with respect to each of the~topologies on $\sigma(\Aa_0)$ considered above). Indeed, it follows from another Lusin--Souslin theorem (see \cite[Thm.~15.1]{kechris}) that since $\Phi$ is continuous and injective, $\Phi(U)$ is Borel for every open set $U\subseteq F$. This means that $\Phi^{-1}$ is Borel.

Fix any nonnegative functions $f_1,f_2\in C_b(F)$ with $f_1f_2=0$. Take an~increasing approximate unit $(h_n)_{n=1}^\infty$ for $C_0(\sigma(\Aa_0))$ such that $0\leq h_n\leq 1$ for each $n\in\N$. Since $f_i\circ\Phi^{-1}$ belongs to $B(\sigma(\Aa_0))$, a~subalgebra of $C_0(\sigma(\Aa_0))^\dast$, we have $h_n\!\cdot(f_i\circ\Phi^{-1}) \nearrow f_i\circ\Phi^{-1}$ pointwise as $n\to\infty$ (for $i=1,2$).\footnote{If $(u_\lambda)$ is an~increasing approximate unit of a~\cs-algebra $\Aa$, then it converges $\sigma$-strongly to the identity in the~enveloping von Neumann algebra $\Aa^\dprime$ of $\Aa$ (see \cite[III.5.2.11]{blackadar}).} 

Now, by transfinite induction, we argue that for every Borel function $\Psi\colon \sigma(\Aa_0)\to F$ and each $n\in\N$ we have
\begin{equation}\label{trans_Psi}
\norm{\phi^\dast(h_n\!\cdot\!(f_1\circ\Psi))\phi^\dast(h_n\!\cdot\!(f_2\circ\Psi))}\leq\e\norm{f_1}_\infty\norm{f_2}_\infty.
\end{equation}
Indeed, by the Lebesgue--Hausdorff--Banach theorem mentioned above, we infer that there is an~ordinal $1\leq\xi<\omega_1$ for which $\Psi\in B_\xi(\sigma(\Aa_0),F)$. If $\xi=1$, we appeal to Lemma~\ref{first_class_L} to produce sequences $(g_{m,i})_{m=1}^\infty$ of nonnegative simple functions assuming each positive value on a~$\Sigma_2^0(\sigma(\Aa_0))$-set and such that: $g_{m,i}\xrightarrow[]{\quad}h_n\!\cdot\!(f_i\circ\Psi)$ pointwise, $g_{m,i}^{-1}(0,\infty)\subseteq (f_i\circ\Psi)^{-1}(0,\infty)$ and $\norm{g_{m,i}}_\infty\leq\norm{f_i}_\infty$ for $m\in\N$, $i=1,2$. In view of {\it Claim~1}, for every $m\in\N$ we have
$$
\norm{\phi^\dast(g_{m,1})\phi^\dast(g_{m,2})}\leq\e\norm{g_{m,1}}_\infty\norm{g_{m,2}}\leq\e\norm{f_1}_\infty\norm{f_2}_\infty,
$$
whence inequality \eqref{trans_Psi} follows by passing to limit as $m\to\infty$. 

Next, assume that $1<\xi<\omega_1$ and inequality \eqref{trans_Psi} holds true whenever $\Psi$ is of Baire class $\eta$ with $\eta<\xi$. Any $\Psi\in B_\xi(\sigma(\Aa_0),F)$ is the~pointwise limit of a~sequence $(\Psi_m)_{m=1}^\infty$ for some $\Psi_m\in B_{\xi_m}(\sigma(\Aa_0),F)$ with $\xi_m<\xi$ ($m\in\N$). By induction hypothesis, estimate \eqref{trans_Psi} is valid with $\Psi_m$ in the~place of $\Psi$, for any $m,n\in\N$. Once again, passing to limit as $m\to\infty$ we obtain the~announced assertion. 

Having proved inequality \eqref{trans_Psi}, we pass to limit once more, this time as $n\to\infty$, and conclude that $\varphi$ is $\eoz{\e}$. This completes the~proof of {\it Claim~2}.

\vspace*{2mm}\noindent
{\it Claim 3. }Assume $g_1,g_2\in C_b(F)$ are nonnegative simple functions such that $g_1g_2=0$ and $g_i^{-1}(\{t\})\in\Pi_1^0(F)$ for all $t>0$, $i=1,2$. Then $\n{\varphi^\dast(g_1)\varphi^\dast(g_2)}\leq\e\norm{g_1}_\infty\norm{g_2}_\infty$.

\vspace*{1mm}
We argue as in the first part of the~proof of {\it Claim~1}, where it was shown that $\mathfrak{D}$ contains all the~closed sets. Replacing $\sigma(\Aa_0)$ by $F$ the~argument goes through {\it mutatis mutandis}, except the fact that the~present case is even easier, as we do not need to care about compact supports.

\vspace*{1mm}
Finally, we are prepared to complete the proof. Fix any $f_1,f_2\in\mathcal{S}$ with $f_1f_2=0$. Recall that the continuous bijection $\Phi\colon F\to\sigma(\Aa_0)$ was chosen according to an~arbitrarily fixed disjoint finite family $\mathcal{F}\subset \mathfrak{B}(\sigma(\Aa_0))$. Now, we take $\mathcal{F}\coloneqq \mathsf{D}(f_1)\cup\mathsf{D}(f_2)$ and choose $\Phi$ correspondingly. Define $g_i=f_i\circ\Phi$ for $i=1,2$. These are Borel step functions on $F$.

By definition, $\varphi=\phi^\dast\circ\Gamma$, where $\Gamma\colon C_b(F)\to B(\sigma(\Aa_0))$ is the~composition operator $\Gamma f=f\circ\Phi^{-1}$. Recall that, due to the~Lusin--Souslin theorem, $\{\Phi^{-1}(A)\colon A\in\mathcal{F}\}$ is a~collection of mutually disjoint closed subsets of $F$. This means that $g_i^{-1}(\{t\})\in \Pi_1^0(F)$ for $i=1,2$ and every $t>0$. As we have seen in the~proof of {\it Claim~1}, such functions can be approximated pointwise by (uniformly bounded) sequences of continuous functions. Therefore, by the~weak$^\ast$-to-weak$^\ast$ continuity of $\Gamma^\dast$, we have $\Gamma^\dast g_i=g_i\circ\Phi^{-1}=f_i$ ($i=1,2$). Hence,
$$
\varphi^\dast(g_i)=\phi^{\ast\ast\ast\ast}\circ\Gamma^\dast(g_i)=\phi^{\ast\ast\ast\ast}(f_i)=\phi^\dast(f_i)\quad\mbox{for }i=1,2.
$$
In view of {\it Claim~3}, we thus have
$$
\norm{\phi^\dast(f_1)\phi^\dast(f_2)}=\norm{\varphi^\dast(g_1)\varphi^\dast(g_2)}\leq\e\norm{g_1}_\infty\norm{g_2}_\infty=\e\norm{f_1}_\infty\norm{f_2}_\infty.
$$
Consequently, we have shown that $\mathfrak{D}=\mathfrak{B}(\sigma(\Aa_0))$. 

It remains to notice that every nonnegative bounded Borel function on $\sigma(\Aa_0)$ is the~pointwise limit of an~increasing sequence of Borel simple funtions. Therefore, in a~similar fashion as several times above, we conclude that $\phi^\dast$ is $\eoz{\e}$ on $B(\sigma(\Aa_0))$.
\end{proof}

\section{Almost commutation relations}
\noindent
It follows from Theorems \ref{wolff_thm} and \ref{WZ_thm} that the range of both a~disjointness preserving and an~order zero operator $\p$ on a~unital \cs-algebra $\Aa$ is contained in the commutant of $\p(1_\Aa)$. As we will see, this has some approximate counterparts, despite two main disadvantages. 

The first one is the well-known fact that, in general, two almost commuting operators on a~Hilbert space may not be `sufficiently' close to commuting operators. In fact, this can happen for matrices---Choi~\cite{choi} constructed, for each $n\in\N$, matrices $A,B\in \mathbb{M}_n(\C)$ satisfying $\n{A}=1-\frac{1}{n}$, $\n{B}\leq 1$, $\n{[A,B]}\leq\frac{2}{n}$, yet $\n{A-A^\prime}+\n{B-B^\prime}\geq 1-\frac{1}{n}$ for every pair $A^\prime, B^\prime\in\mathbb{M}_n(\C)$ with $[A^\prime,B^\prime]=0$. On a~positive side, the famous Lin's theorem \cite{lin} says that for any $\delta>0$ there is $\e>0$ such that for any self-adjoint matrices $A,B\in\mathbb{M}_n(\C)$ with $\n{A}, \n{B}\leq 1$ and $\n{[A,B]}<\delta$ there exists a~pair of commuting self-adjoint matrices $A^\prime, B^\prime\in\mathbb{M}_n(\C)$ such that $\n{A-A^\prime}+\n{B-B^\prime}<\e$. However, we cannot guarantee that one can take e.g. $A^\prime=A$, so an~almost commutation relation between $A$ and $B$ does not automatically imply any~similar relation between $P(A)$ and $B$ for $P$ being a~polynomial.

Another disadvantage is, therefore, that given an~$\eoz{\e}$ map $\phi$ on a~unital \cs-algebra $\Aa$, an~upper bound on the norm of the commutator $[P(\phi(1_\Aa)),\phi(x)]$ grow quite rapidly when $\mathrm{deg}\,P\to\infty$ (see Proposition~\ref{comm_polynomial} below). Nonetheless, there are some situations where we can obtain a~`uniform' estimate up to a~supremum norm of $P$ (see Propositions~\ref{alg_prop} and \ref{hZ_prop}) and get almost commutation relations with all spectral projections of $\phi(1_\Aa)$ (as in Lemma~\ref{sp_proj}) which will be important later on.

We denote by $\tau$ the~operator on the~space $\C[z]$ of complex polynomials acting as the left-shift on coefficients, {\it i.e.} $\tau P(z)=z^{-1}(P(z)-P(0))$. Given any $\eoz{\e}$ map $\phi$ on a~unital \cs-algebra $\Aa$, we define for every $P\in\C[z]$ with $\mathrm{deg} P=N>0$ a~nonnegative number $\Theta_N(P)$ recursively as follows:
\begin{equation}\label{norm_rec}
\left\{\begin{array}{lc}
\Theta_1(az+b)=8\abs{a}\e & \\
\Theta_N(P)=8\n{\phi}_+\Theta_{N-1}(\tau P)+8\n{\tau P(\phi(1_\Aa))}\e & \mbox{for }N\geq 2,
\end{array}\right.
\end{equation}
where 
$$
\n{\phi}_+\coloneqq\sup\{\n{\phi(x)}\colon x\in\Aa_+,\, \n{x}\leq 1\}.
$$
Obviously, the Jordan decomposition yields $\n{\phi}\!\leq 4\n{\phi}_+$, however, we prefer to be as precise as possible in estimate \eqref{P(h)_ineq} below, in order to eventually obtain better constants in Lemma~\ref{h_vs_psi_pos} and, ultimately, in Corollary~\ref{main_corollary}.

\begin{proposition}\label{comm_polynomial}
Let $\Aa$ and $\Bb$ be \cs-algebras and assume $\Aa$ is unital. Let $\phi\in\LL(\Aa,\Bb)$ be an~$\eoz{\e}$ map with $h\coloneqq \phi(1_\Aa)$ and $P\in\C[z]$ be a~complex polynomial with $\mathrm{deg} P=N>0$. Then we have
\begin{equation}\label{P(h)_ineq}
\n{P(h)\phi(x)-\phi(x)P(h)}\leq \Theta_N(P)\n{x}\quad\mbox{for every }x\in\Aa_+.
\end{equation}
\end{proposition}

\begin{proof}
To start the induction we show the inequality
\begin{equation}\label{comm_h}
\n{h\phi(x)-\phi(x)h}\leq 8\e\n{x}\quad\mbox{for every }x\in\Aa_+.
\end{equation}
By homogeneity, we can assume that $0\leq x\leq 1_\Aa$. We use the notation from the proof of Proposition~\ref{P_almostJordan}. Observe that in the algebra \cs$(x,1_\Aa)\cong C(\sigma(x))$, for each $n\in\N$, we have
$$
\ind_{\sigma(x)}=\sum_{k=0}^{n-1}\ind_{X_{k,n}}\quad\mbox{and }\quad h=\sum_{k=0}^{n-1}\phi^{\ast\ast}(\ind_{X_{k,n}}).
$$
Therefore, for any $f\in C(\sigma(x))$ with $\n{f}_\infty\leq 1$ we have
\begin{equation*}
\begin{split}
h\cdot\sum_{k=0}^{n-1} f(x_{k,n})\phi^{\ast\ast}(\ind_{X_{k,n}})=\sum_{k=0}^{n-1} &f(x_{k,n})(\phi^{\ast\ast}(\ind_{X_{k,n}}))^2\\
&+\sum_{0\leq j\not=k<n}f(x_{k,n})\phi^{\ast\ast}(\ind_{X_{j,n}})\phi^{\ast\ast}(\ind_{X_{k,n}}).
\end{split}
\end{equation*}
As it was proved above ({\it cf.} inequality \eqref{a_estimate}), the norm of the last summand is at most $4\e$. The reversed product can be represented similarly, so subtracting these equations we get
$$
\Biggl\|h\cdot\Bigl(\sum_{k=0}^{n-1} f(x_{k,n})\phi^{\ast\ast}(\ind_{X_{k,n}})\Bigr)-\Bigl(\sum_{k=0}^{n-1} f(x_{k,n})\phi^{\ast\ast}\!(\ind_{X_{k,n}})\Bigr)\cdot h\Biggr\|\leq 8\e.
$$
In view of formula \eqref{phi(f)}, we obtain inequality \eqref{comm_h}, that is, our assertion for $\mathrm{deg} P=1$.

Now, fix any polynomial $P\in\C[z]$ with $\mathrm{deg} P=N\geq 2$ and assume that the desired inequality holds true for all $\eoz{\e}$ maps and all complex polynomials of degree smaller than $N$. Fix any $x\in \Aa_\sa$ such that $0\leq x\leq 1_\Aa$ and write $P(z)=\sum_{j=0}^N a_jz^j$. Again, we use the~notation from the proof of Proposition~\ref{P_almostJordan}. 

For any $f\in C(\sigma(x))$, define
\begin{equation*}
\begin{split}
\beta_n(f) & \coloneqq P(h)\Biggl(\sum_{l=0}^{n-1}f(x_{l,n})\phi^\dast(\ind_{X_{l,n}})\Biggr)\\
& \,= \sum_{j=0}^N a_j\sum_{k=0}^{n-1}f(x_{l,n})\Biggl(\sum_{k=0}^{n-1}\phi^\dast(\ind_{X_{k,n}})\Biggr)^j\phi^\dast(\ind_{X_{l,n}})\\
& \,=\sum_{j=0}^N a_j\sum_{0\leq k_0,k_1,\ldots,k_j<n}\!\! f(x_{k_j,n})\phi^\dast(\ind_{X_{k_0,n}})\phi^\dast(\ind_{X_{k_1,n}})\cdot\ldots\cdot \phi^\dast(\ind_{X_{k_j,n}}).
\end{split}
\end{equation*}
Similarly, we calculate
\begin{equation*}
\begin{split}
\gamma_n(f) &\coloneqq \Biggl(\sum_{k=0}^{n-1}f(x_{k,n})\phi^\dast(\ind_{X_{k,n}})\Biggr)P(h)\\
& \,=\sum_{j=0}^N a_j\sum_{0\leq k_0,k_1,\ldots,k_j<n}\!\! f(x_{k_0,n})\phi^\dast(\ind_{X_{k_0,n}})\phi^\dast(\ind_{X_{k_1,n}})\cdot\ldots\cdot \phi^\dast(\ind_{X_{k_j,n}}).
\end{split}
\end{equation*}
Both these expressions $\beta_n(f)$ and $\gamma_n(f)$ have the~same part corresponding to $(j+1)$-tuples $(k_0,k_1,\ldots,k_j)$ for which $k_0=k_j$. After reducing this common part we obtain that $\beta_n(f)-\gamma_n(f)=\beta_n^\prime(f)-\gamma_n^\prime(f)$, where $\beta_n^\prime(f)$ is defined as
\begin{equation*}
\begin{split}
\sum_{j=1}^N a_j\!\!\sum_{0\leq k_0\neq k_j<n}\!\!\!\!\! & f(x_{k_j,n}) \phi^\dast(\ind_{X_{k_0,n}})\Biggl(\sum_{k_1,\ldots,k_{j-1}=0}^{n-1}\!\!\!\phi^\dast(\ind_{X_{k_1,n}})\cdot\ldots\cdot \phi^\dast(\ind_{X_{k_{j-1},n}})\!\Biggr)\phi^\dast(\ind_{X_{k_j,n}})\\
&=\sum_{0\leq k\neq l<n}\!\!\!\!\! f(x_{l,n})\phi^\dast(\ind_{X_{k,n}})\Biggr(\sum_{j=1}^N a_jh^{j-1}\Biggl)\phi^\dast(\ind_{X_{l,n}})\\
&=\sum_{0\leq k\neq l<n}\!\!\!\!\! f(x_{l,n}) \phi^\dast(\ind_{X_{k,n}})\,\tau P(h)\phi^\dast(\ind_{X_{l,n}})
\end{split}
\end{equation*}
and $\gamma_n^\prime(f)$ is defined by an~analogous formula with $f(x_{l,n})$ replaced by $f(x_{k,n})$. (Notice that for $j=1$ the term in brackets is meant to be the identity in $\Bb^\dast$.)

We {\it claim} that for all $n\in\N$ and $f\in C(\sigma(x))$ with $\norm{f}_\infty\leq 1$ we have
\begin{equation}\label{bet_gam}
\norm{\beta_n^\prime(f)},\,\norm{\gamma_n^\prime(f)}\leq (4-2^{3-n})(\n{\phi}_+\Theta_{N-1}(\tau P)+\n{\tau P(h)})\e.
\end{equation}
We consider the inequality for $\beta_n^\prime(f)$; the~other one just requires changing one index. By virtue of Proposition~\ref{Borel_lemma}, the~operator $\phi^\dast\restriction_{B(\sigma(x))}$ is $\eoz{\e}$ Therefore, we can apply our induction hypothesis to $\phi^\dast$ and the polynomial $\tau P$ of degree $N-1$. As before, $\Pi$ stands for the~collection of nontrivial ordered partitions of $\{0,1,\ldots,n-1\}$. Given any $(K,L)\in\Pi$, we thus obtain
\begin{equation*}
\begin{split}
\Biggl\|\phi^\dast\Bigl( & \sum_{k\in K}\ind_{X_{k,n}}\Bigr)\,\tau P(h)\phi^\dast\Bigl(\sum_{l\in L}f(x_{l,n})\ind_{X_{l,n}}\Bigr)\Biggr\|\\
&\leq \Biggl\|\phi^\dast\Bigl(\sum_{k\in K}\ind_{X_{k,n}}\Bigr)\Biggr\|\cdot\Biggl\|\tau P(h)\phi^\dast\Bigl(\sum_{l\in L}f(x_{l,n})\ind_{X_{l,n}}\Bigr)-\phi^\dast\Bigl(\sum_{l\in L}f(x_{l,n})\ind_{X_{l,n}}\Bigr)\tau P(h)\Biggr\|\\
&\hspace*{179pt} +\Biggl\|\phi^\dast\Bigl(\sum_{k\in K}\ind_{X_{l,n}}\Bigr)\phi^\dast\Bigl(\sum_{l\in L}f(x_{l,n})\ind_{X_{l,n}}\Bigr)\tau P(h)\Biggr\|\\
&\leq \n{\phi}_+\Theta_{N-1}(\tau P)+\n{\tau P(h)}\e.
\end{split}
\end{equation*}
(Notice that the estimate by $\n{\phi}_+$ follows from the fact that $\phi^\dast$ is weak$^\ast$-to-weak$^\ast$ continuous and $\sum_{k\in K}\ind_{X_{k,n}}$ is a~pointwise limit of continuous positive functions.) Summing up over all partitions we get
\begin{equation*}
\begin{split}
\sum_{(K,L)\in\Pi}\Biggl\|\Bigl(\sum_{k\in K}f(x_{k,n})\phi^\dast(\ind_{X_{k,n}})\Bigr)\, &\tau P(h)\Bigl(\sum_{l\in L}\phi^\dast(\ind_{X_{l,n}})\Bigr)\Biggr\|\\
&\leq (2^n-2)(\n{\phi}_+\Theta_{N-1}(\tau P)+\n{\tau P(h)}\e).
\end{split}
\end{equation*}
For any fixed pair $(k,l)$ with $0\leq k\neq l<n$, the~number of those partitions $(K,L)\in\Pi$ for which $k\in K$ and $l\in L$ equals $2^{n-2}$. Consequently, \eqref{bet_gam} follows from the~triangle inequality.

By Lebesgue's theorem, we have $\beta_n(\mathrm{id}_{\sigma(x)})\xrightarrow[]{\,w\ast\,}P(h)\phi(x)$ and $\gamma_n(\mathrm{id}_{\sigma(x)})\xrightarrow[]{\,w\ast\,}\phi(x)P(h)$. Hence, inequality \eqref{bet_gam} yields
\begin{equation*}
\begin{split}
\norm{P(h)\phi(x)-\phi(x)P(h)} &\leq \liminf_{n\to\infty}\norm{\beta_n(\mathrm{id}_{\sigma(x)})-\gamma_n(\mathrm{id}_{\sigma(x)})}\\
&=\liminf_{n\to\infty}\norm{\beta_n^\prime(\mathrm{id}_{\sigma(x)})-\gamma_n^\prime(\mathrm{id}_{\sigma(x)})}\\
&\leq 8(\n{\phi}_+\Theta_{N-1}(\tau P)+\n{\tau P(h)}\e)=\Theta_N(P).
\end{split}
\end{equation*}
This completes the induction.
\end{proof}

Obviously, the estimate given in Proposition~\ref{comm_polynomial} is meaningless if $\Theta_N(P)$ is too large, e.g. greater than $2\n{\phi}_+\n{P(h)}$. However, in some cases inequality \eqref{P(h)_ineq} can be transformed so that we obtain `almost commutation' relations between $\phi(x)$ and spectral projections of $h$. In what follows, we shall see two such cases. The first one happens when $h$ is an algebraic element, i.e. $P(h)=0$ for some monic polynomial $P\in\C[z]$. The second one is more subtle and relies on a~deep result by Alaminos, Extremera and Villena \cite{alaminos} mentioned before.

\begin{proposition}\label{alg_prop}
For all $N\in\N$ and $M>0$ there exists $C=C(N,M)<\infty$ with the following property: If $\Aa$ and $\Bb$ are \cs-algebras, $\Aa$ is unital, $\phi\in\LL(\Aa,\Bb)$ is a~nonzero $\eoz{\e}$ map and $h\coloneqq \phi(1_\Aa)$ is a~self-adjoint algebraic element of order at most $N$ such that $\n{\phi}_+\leq M\n{h}$, then for every $P\in\C[z]$ with $P(0)=0$ and $x\in\Aa_+$ we have
\begin{equation*}
\n{P(h)\phi(x)-\phi(x)P(h)}\leq \frac{C\e}{\n{h}}\sup_{\abs{z}=\n{h}}\abs{P(z)}\cdot\n{x}.
\end{equation*}
\end{proposition}
\begin{proof}
First, observe that if $h\in\Bb$ is an algebraic element of order at most $N$, then the commutator $[P(h),\phi(x)]$ is the same as $[Q(h),\phi(x)]$ with some $Q\in\C[z]$ satisfying $\mathrm{deg}\,Q<N$. Therefore, we can assume that $P$ satisfies $P(0)=0$ and $\mathrm{deg}\,P<N$.

For each $1\leq j<N$ define a function $\w\Theta_j$ on the set of complex polynomials of degree $j$ by $\w\Theta_j(P)\e=\Theta_j(P)$, where $\Theta_j$ are defined by formulas \eqref{norm_rec}. Notice that
$$
\w\Theta_1(z)=8=\frac{8}{\n{h}}\sup_{z\in\sigma(h)}\abs{z},
$$
hence our inequality holds true for $\mathrm{deg}\,P=1$ in view of Proposition~\ref{comm_polynomial}. 

Fix any $1<j<N$ and any $P\in\C[z]$ with $\mathrm{deg}\,P=j$. Let also $\phi$ and $h$ be as above. Assume that for every polynomial $Q\in\C[z]$ with $\mathrm{deg}\,Q=j-1$ we have 
$$
\w\Theta_{j-1}(Q)\leq C_{j-1}\sup_{\abs{z}=\n{h}}\abs{Q(z)}.
$$
Plainly, formulas \eqref{norm_rec} yield
$$
\w\Theta_j(P)=8\n{\phi}_+\w\Theta_{j-1}(\tau P)+8\sup_{z\in\sigma(h)}\abs{\tau P(z)}
$$
and hence our inductive hypothesis implies that
\begin{equation*}
\begin{split}
\w\Theta_j(P) &\leq 8(C_{j-1}\n{\phi}_++1)\sup_{\abs{z}=\n{h}}\abs{\tau P(z)}\\
&=8\Big(\frac{\n{\phi}_+}{\n{h}}C_{j-1}+\frac{1}{\n{h}}\Big)\sup_{\abs{z}=\n{h}}\abs{P(z)}\leq 8\Big(MC_{j-1}+\frac{1}{\n{h}}\Big)\sup_{\abs{z}=\n{h}}\abs{P(z)}.
\end{split}
\end{equation*}
Consequently, if we define a~sequence $(C_1,\ldots,C_{N-1})$ recursively by
$$
C_1=\frac{8}{\n{h}}\quad\mbox{ and }\quad C_j=8\Big(MC_{j-1}+\frac{1}{\n{h}}\Big)\quad\mbox{for }1<j<N,
$$
then in view of Proposition~\ref{comm_polynomial} we have
$$
\n{P(h)\phi(x)-\phi(x)P(h)}\leq C_j\,\e\!\sup_{\abs{z}=\n{h}}\abs{P(z)}\cdot\n{x}\quad\mbox{for every }x\in\Aa_+.
$$
It is easily seen that $C_j=\n{h}^{-1}\sum_{i=0}^{j-1} 8^{i+1}M^i$ for each $1\leq j<N$. Hence, putting $C=\sum_{i=0}^{N-2} 8^{i+1}M^i$ we obtain the desired estimate. 
\end{proof}
\begin{remark}
We will later show that the inequality $\n{\phi}_+\leq M\n{h}$, with some absolute constant $M<\infty$, can be safely assumed for $\phi$ being self-adjoint, as otherwise the norm of $\phi$ is small (see Corollaries~\ref{h_vs_psi} and \ref{h_vs_psi_pos}). This observation will be used several times in the proof of Theorem~\ref{L_decomposition}.
\end{remark}

\begin{lemma}\label{phi_P_phi}
Let $\Aa$ and $\Bb$ be unital \cs-algebras and $\phi\in\LL(\Aa,\Bb)$ be an~$\eoz{\e}$ map with $h\coloneqq \phi(1_\Aa)$ and satisfying $\n{\phi}>\e^{1/2}$. Assume also that the \cs-subalgebra of $\Bb$ generated by $h$ lies in the range of the center of $\Aa$ under $\phi$, that is,
$$
\mathrm{C}^\ast(h)\subseteq \phi(\mathcal{Z}(\Aa)).
$$
Then, there exists a~constant $M>0$ such that for all $x,y\in\Aa_+$ with $xy=0$ and every polynomial $P\in\C[z]$ we have
\begin{equation}\label{phi_P_ineq}
\n{\phi(x)P(h)\phi(y)}\leq M\n{\phi}\big(\n{\phi}^{15/8}\mathcal{O}(\e^{1/16})+24\e\big)\n{P(h)}\n{x}\n{y}.
\end{equation}
Moreover, in the case where $\Aa$ is commutative and $\phi$ is surjective, one can take $M$ to be the openness index of $\phi$.
\end{lemma}
\begin{proof}
Similarly as in the proof of the Open Mapping Theorem, we note that there exists $r>0$ such that
\begin{equation}\label{openn}
    rB_\Bb\cap\mathrm{C}^\ast(h)\subset\oo{\phi(B_{\mathcal{Z}(\Aa)})},
\end{equation}
where $B_X$ stands for the closed unit ball of $X$. (Recall that $\mathcal{Z}(\Aa)$ is a~closed subspace of $\Aa$.) Indeed, since
$$
\oo{\phi\Big(\frac{1}{2}B_{\mathcal{Z}(\Aa)}\Big)}-\oo{\phi\Big(\frac{1}{2}B_{\mathcal{Z}(\Aa)}\Big)}\subseteq \oo{\phi\Big(\frac{1}{2}B_{\mathcal{Z}(\Aa)}\Big)-\phi\Big(\frac{1}{2}B_{\mathcal{Z}(\Aa)}\Big)}\subseteq \oo{\phi(B_{\mathcal{Z}(\Aa)})},
$$
it is enough to show that $\oo{\phi(\frac{1}{2}B_{\mathcal{Z}(\Aa)})}$ has nonempty interior relatively to $\mathrm{C}^\ast(h)$. As we have
$$
\mathrm{C}^\ast(h)\subseteq \phi(\mathcal{Z}(\Aa))=\bigcup_{k=1}^\infty k\phi\Big(\frac{1}{2}B_{\mathcal{Z}(\Aa)}\Big),
$$
it follows from the Baire Category Theorem that for some $k\in\N$ the set $k\phi(\frac{1}{2}B_{\mathcal{Z}(\Aa)})$ is not nowhere dense relatively to $\mathrm{C}^\ast(h)$. The same is then true for $\phi(\frac{1}{2}B_{\mathcal{Z}(\Aa)})$ and our claim follows.

Now, by \eqref{openn}, we infer that there exists $M>0$ such that for every $P\in\C[z]$ and any $\delta>0$ there exists $v\in\mathcal{Z}(\Aa)$ satisfying $\n{\phi(v)-P(h)}\leq\delta$ and $\n{v}\leq M\n{P(h)}$. At this point, note that if $\mathcal{Z}(\Aa)=\Aa$ and $\phi$ is surjective, the Open Mapping Theorem guarantees that we can take $\delta=0$ and $M=\mathrm{op}(\phi)$.

Fix $x$ and $y$ as above and consider the~commutative \cs-subalgebra $\mathrm{C}^\ast(x,y,1_\Aa)$ of $\Aa$ generated by $\{x,y,1_\Aa\}$. Regarding any its elements $u, v$ as continuous functions, we easily see that the~condition $uv=0$ is equivalent to $u\!\perp\! v$. Therefore, by Lemma~\ref{L_eoz}(b), we have
$$
u,v\in \mathrm{C}^\ast(x,y,1_\Aa),\,\, uv=0\,\,\xRightarrow[]{\phantom{xxx}}\,\, \n{\phi(u)\phi(v)}\leq 16\e\n{x}\n{y}.
$$
Consider the bilinear map $\Phi(u,v)=\phi(u)\phi(v)$ and put $K\coloneqq 16\n{\phi}^2>\max\{\n{\Phi},16\e\}$ (recall that $\n{\phi}>\e^{1/2}$). In view of Theorem~\ref{alaminos_thm}, we have
\begin{equation}\label{AEV11}
\n{\phi(uv)\phi(w)-\phi(u)\phi(vw)}\leq \eta\n{u}\n{v}\n{w}\quad\mbox{for all }u,v,w\in \mathrm{C}^\ast(x,y,1_\Aa),
\end{equation}
where
\begin{equation}\label{eta_def}
\eta=4\Big(\frac{17^2}{3}+1\Big)^2 K^{1/2}\e^{1/2}\Big(2+\zeta\Big(\frac{16\e}{K}\Big)\Big)+K\zeta\Big(\frac{16\e}{K}\Big).
\end{equation}
By appealing to Lemma~\ref{zeta_O}, it can be easily verified that $\eta$, defined by formula \eqref{eta_def} as a~function of $\e$, satisfies the~estimate
$$
\eta(\e)\leq M\big(\n{\phi}\e^{1/2}+\n{\phi}^{7/8}\e^{9/16}+\n{\phi}^{15/8}\e^{1/16}\big),
$$
with some absolute constant $M<\infty$. Since $\n{\phi}>\e^{1/2}$, we have $\n{\phi}\e^{1/2}<\n{\phi}^{15/8}\e^{1/16}$ and $\n{\phi}^{7/8}\e^{9/16}<\n{\phi}^{15/8}\e^{1/16}$, whence 
\begin{equation}\label{AEV12}
\eta(\e)=\n{\phi}^{15/8}\mathcal{O}(\e^{1/16}).
\end{equation}

Fix $P\in\C[z]$ and consider any $v\in\mathcal{Z}(\Aa)$ with $\n{v}\leq M\n{P(h)}$. For simplicity, assume that $\n{x},\n{y}\leq 1$. Notice that $xvy=0=yxv$, as well as $(xv)^\ast y=v^\ast x^\ast y=0=xvy^\ast$, i.e. $xv\perp y$. Hence, Lemma~\ref{L_eoz}(b) yields $\n{\phi(xv)\phi(y)}\leq 16\e\n{v}$. In view of \eqref{AEV11}, \eqref{AEV12} and Proposition~\ref{comm_polynomial}, we have
\begin{equation*}
    \begin{split}
        \n{\phi(x)\phi(v)\phi(y)} & \leq \n{\phi(x)\phi(v)-\phi(xv)h}\n{\phi}+\n{\phi(xv)h\phi(y)}\\
        & \leq \eta(\e)\n{\phi}\n{v}+\n{\phi(xv)\phi(y)h}+\n{\phi}\n{v}\n{h\phi(y)-\phi(y)h}\\
        & \leq \n{\phi}\n{v}\big(\eta(\e)+16\e+8\e\big)\\
        & \leq M\n{\phi}\big(\n{\phi}^{15/8}\mathcal{O}(\e^{1/16})+24\e\big)\n{P(h)}.
    \end{split}
\end{equation*}
Since $v\in\mathcal{Z}(\Aa)$ can be chosen so that $\phi(v)$ is arbitrarily close to $P(h)$, we obtain inequality~\eqref{phi_P_ineq}.
\end{proof}

\begin{proposition}\label{hZ_prop}
Let $\Aa$ and $\Bb$ be \cs-algebras and assume $\Aa$ is unital. Let $\phi\in\LL(\Aa,\Bb)$ be an~$\eoz{\e}$ map with $0\neq h\coloneqq \phi(1_\Aa)$ and satisfying
$$
\mathrm{C}^\ast(h)\subseteq \phi(\mathcal{Z}(\Aa)).
$$
Then, there exists a~constant $M>0$ such that for every polynomial $P\in\C[z]$ with $P(0)=0$ and every $x\in\Aa_+$ we have
\begin{equation*}\label{hZ_ineq}
\n{P(h)\phi(x)-\phi(x)P(h)}\leq\max\Bigg\{2\e^{1/2}, 8M\frac{\n{\phi}}{\n{h}}\big(\n{\phi}^{15/8}\mathcal{O}(\e^{1/16})+24\e\big)\Bigg\}\!\sup_{\abs{z}=\n{h}}\abs{P(z)}\!\cdot\!\n{x}.
\end{equation*}
Moreover, in the case where $\Aa$ is commutative (i.e. $\mathcal{Z}(\Aa)=\Aa$) and $\phi$ is surjective, one can take $M$ to be the openness index of $\phi$.
\end{proposition}
\begin{proof}
First, observe that if $\n{\phi}\leq\e^{1/2}$, then the above inequality is trivial. So, suppose that $\n{\phi}>\e^{1/2}$ and take any $x\in\Aa_+$ with $0\leq x\leq 1_\Aa$. As we have seen in the proof of Proposition~\ref{comm_polynomial}, the commutator $[P(h),\phi(x)]$ is the weak$^\ast$ limit of $\beta_n(\mathrm{id}_{\sigma(x)})-\gamma_n(\mathrm{id}_{\sigma(x)})=\beta_n^\prime(\mathrm{id}_{\sigma(x)})-\gamma_n^\prime(\mathrm{id}_{\sigma(x)})$, where 
$$
\beta_n^\prime(\mathrm{id}_{\sigma(x)})=\sum_{0\leq k\neq l<n} x_{l,n}\phi^\dast(\ind_{X_{k,n}})\tau P(h)\phi^\dast(\ind_{X_{l,n}})
$$
and $\gamma_n^\prime(\mathrm{id}_{\sigma(x)})$ is defined analogously with $x_{k,n}$ instead of $x_{l,n}$. Applying Lemma~\ref{phi_P_phi} to the $\eoz{\e}$ map $\phi^\dast$ we infer that for every partition $(K,L)\in\Pi$ we have
\begin{equation*}
\begin{split}
\Biggl\|\phi^\dast\Bigl(\sum_{k\in K}\ind_{X_{k,n}}\Bigr)\,\tau P(h) &\phi^\dast\Bigl(\sum_{l\in L}x_{l,n}\ind_{X_{l,n}}\Bigr)\Biggr\|\\
&\leq M\n{\phi}\big(\n{\phi}^{15/8}\mathcal{O}(\e^{1/16})+24\e\big)\n{\tau P(h)}.
\end{split}
\end{equation*}
Summing over all partitions, using the triangle inequality and Lebesgue's theorem as in the proof of Proposition~\ref{comm_polynomial}, we obtain
$$
\n{P(h)\phi(x)-\phi(x)P(h)}\leq 8M\n{\phi}\big(\n{\phi}^{15/8}\mathcal{O}(\e^{1/16})+24\e\big)\n{\tau P(h)}.
$$
It remains to observe that
\begin{equation*}
\n{\tau P(h)}=\sup_{z\in\sigma(h)}\abs{\tau P(z)}\leq\sup_{\abs{z}=\n{h}}\abs{\tau P(z)}=\frac{1}{\n{h}}\sup_{\abs{z}=\n{h}}\abs{P(z)}.
\end{equation*}
The last sentence of our assertion follows from the fact that $M$ is the same constant as the one stemming from Lemma~\ref{phi_P_phi}.
\end{proof}

The following simple lemma which guarantees `almost commutation' relations with spectral projections will be used in our decomposition result in Section~7.
\begin{lemma}\label{sp_proj}
Let $S,T\in\LL(\Hh)$ be operators on a~complex Hilbert space $\Hh$ with $S$ being self-adjoint. Suppose $\delta, R>0$ are such that for every complex polynomial $P\in\C[z]$ with $P(0)=0$ we have 
$$
\n{P(S)T-TP(S)}\leq\delta\sup_{\abs{z}=R}\abs{P(z)}.
$$
Then for every spectral projection $V$ of $S$ we have $\n{VT-TV}\leq\delta$.
\end{lemma}
\begin{proof}
Let $\mathsf{E}$ be the spectral measure of $S$. For any $\xi,\eta\in\Hh$ we denote by $\mathsf{E}_{\xi,\eta}$ the complex Borel measure on $\sigma(S)$ given by $\mathsf{E}_{\xi,\eta}(\omega)=\langle \mathsf{E}(\omega)\xi,\eta\rangle$. Fix arbitrarily $\xi,\eta\in\Hh$ with $\n{\xi}=\n{\eta}=1$, denote $\mu=T^\ast\eta$ and notice that for every $P\in\C[z]$ (with $P(0)=0$) we have
$$
\langle P(S)T\xi,\eta\rangle=\int_{\sigma(S)}P(z)\,\dd\mathsf{E}_{T\xi,\eta}(z)
$$
and
$$
\langle TP(S)\xi,\eta\rangle=\langle P(S)\xi,\mu\rangle =\int_{\sigma(S)}P(z)\,\dd\mathsf{E}_{\xi,\mu}(z).
$$
Therefore, by our assumption,
$$
\Big|\int_{\sigma(S)}P(z)\,\dd\mathsf{E}_{T\xi,\eta}(z)-\int_{\sigma(S)}P(z)\,\dd\mathsf{E}_{\xi,\mu}(z)\Big|\leq\delta\sup_{\abs{z}=R}\abs{P(z)}.
$$
Define $\nu$ to be the complex Borel measure that extends the measure $\mathsf{E}_{T\xi,\eta}-\mathsf{E}_{\xi,\mu}$ to the whole disc $\{\abs{z}\leq R\}$, that is, $\nu(\omega)=(\mathsf{E}_{T\xi,\eta}-\mathsf{E}_{\xi,\mu})(\omega\cap\sigma(S))$ for every Borel subset $\omega$ of $\{\abs{z}\leq R\}$. Then we have
$$
\Big|\int_{\{\abs{z}\leq R\}}P(z)\,\dd\nu(z)\Big|\leq\delta \sup_{\abs{z}=R}\abs{P(z)}
$$
and, of course, the same inequality holds true if we replace $P(z)$ by $\oo{P(z)}$ as the measure $\nu$ is supported only on $\sigma(S)\subset\R$. The~Stone--Weierstrass theorem and the Maximum Modulus Principle thus imply that $\nu$ defines a~linear functional of norm at most $\delta$ on the Banach space of complex continuous functions on $\{\abs{z}\leq R\}$ vanishing at $0$ (equipped with the supremum norm). Hence, $\abs{\nu}(\sigma(S))\leq\delta$ and in particular for every Borel set $\omega\subseteq\sigma(S)$ we have $\abs{\nu(\omega)}\leq\delta$ which means that 
$$
\abs{\langle \mathsf{E}(\omega)T\xi,\eta\rangle-\langle T\mathsf{E}(\omega)\xi,\eta\rangle}\leq\delta.
$$
Since $\xi$ and $\eta$ were arbitrary unit vectors, we obtain $\n{\mathsf{E}(\omega)T-T\mathsf{E}(\omega)}\leq\delta$, as desired.
\end{proof}

\section{The range of an approximately order zero map}
\noindent
In this section, we seek for approximate counterparts of the fact that the range of any disjointness preserving or order zero operator $\p$ on a~unital \cs-algebra $\Aa$ has a~range contained in the closure of $\p(1_\Aa)\{\p(1_\Aa)\}^\prime$, as stated in Theorems~\ref{wolff_thm} and \ref{WZ_thm}. It turns out that the range of any self-adjoint $\eoz{\e}$ map $\psi\colon\Aa\to\Bb$ lies close to the hereditary \cs-subalgebra $\oo{\psi(1_\Aa)\Bb\psi(1_\Aa)}$ of $\Bb$ (it is indeed hereditary as $\psi(1_\Aa)=\psi(1_\Aa)^\ast$; see \cite[Prop.~II.3.4.2]{blackadar}). To show this, we need a~deep result by Aleksandrov and Peller \cite{AP} which says that any $\alpha$-H\"older function on $\R$, with $0<\alpha<1$, is also operator H\"older:
\begin{theorem}[{\cite[Thm.~4.1]{AP}}]\label{AP_thm}
Let $0<\alpha<1$. There exists a~constant $c>0$ depending only on $\alpha$ such that for every $\alpha$-H\"older function $f\colon\R\to\R$ and any self-adjoint operators $A$, $B$ on a~Hilbert space, we have
$$
\n{f(A)-f(B)}\leq c\n{f}_{\Lambda_\alpha(\R)}\n{A-B}^\alpha,
$$
where $$\n{f}_{\Lambda_\alpha(\R)}=\sup_{x\neq y}\frac{\abs{f(x)-f(y)}}{\abs{x-y}^{\alpha}}.$$
\end{theorem}

Our next lemma actually gives a~more precise information about the range of $\psi$ than just that it lies close to $\oo{\psi(1_\Aa)\Bb \psi(1_\Aa)}$. This requires a~use of some well-known results on (weak) polar decompositions in \cs-algebras. There are the ``left-handed",``right-handed" and ``two-sided" versions which we quote below.
\begin{proposition}[{see \cite[II.3.2.1 and II.3.2.2]{blackadar}}]\label{polar}
Let $\Aa$ be a \cs-algebra, $x\in \Aa$ and $a\in \Aa_+$.
\begin{enumerate}[{\rm (i)}]
\item If $x^\ast x\leq a$, then for every $0<\alpha<\frac12$ there exists $u\in\oo{\Aa a}$ satisfying $x=ua^\alpha$ and $\n{u}\leq \n{a^{1/2-\alpha}}$.
\item  If $xx^\ast \leq a$, then for every $0<\alpha<\frac12$ there exists $v\in\oo{a\Aa}$ satisfying $x=a^\alpha v$ and $\n{v}\leq \n{a^{1/2-\alpha}}$.
\end{enumerate}
\end{proposition}
\begin{proposition}[{see \cite[Lemma~2.2.4]{loring}}]\label{polar_twosided}
Let $\Aa$ be a \cs-algebra, $x\in \Aa$ and $a\in \Aa_+$. If $x^\ast x\leq a$ and $xx^\ast\leq a$, then for every $0<\alpha<\frac14$ there exists $d\in \Aa$ satisfying $x=a^\alpha da^\alpha$ and $\n{d}\leq\n{x}^{1/2-2\alpha}$.
\end{proposition}

\begin{lemma}\label{close_to_hereditary}
There exists an absolute constant $K<\infty$ such that the~following holds. Let $\Aa, \Bb$ be \cs-algebras, with $\Aa$ unital, let $\psi\colon \Aa\to \Bb$ be a~self-adjoint $\eoz{\e}$ map and define $h\coloneqq \psi(1_\Aa)$. Then we have
\begin{equation}\label{distance}
\mathrm{dist}\big(\psi(x),\oo{h\Bb h}\,\big)\leq K\n{\psi}^{3/5}\e^{1/5}\n{x}\quad\mbox{for every }x\in \Aa.
\end{equation}
Moreover:
\begin{enumerate}[{\rm (a)}]
\item for all $0<\alpha<\frac{1}{10}$, $x\in \Aa$ there exists $u(x)\in \Bb$ such that $\n{u(x)}\leq 2\n{\psi}^{1-2\alpha}\n{x}$ and
$$
\n{\psi(x)-u(x)(h^2)^\alpha}\leq K\n{\psi}^{3/5}\e^{1/5}\n{x};
$$
\item for all $0<\alpha<\frac{1}{10}$, $x\in \Aa$ there exists $v(x)\in \Bb$ such that $\n{v(x)}\leq 2\n{\psi}^{1-2\alpha}\n{x}$ and
$$
\n{\psi(x)-(h^2)^{\alpha}v(x)}\leq K\n{\psi}^{3/5}\e^{1/5}\n{x};
$$
\item for all $0<\alpha<\frac{1}{20}$, $x\in \Aa$ there exists $d(x)\in \Bb$ satisfying $$\n{d(x)}\leq 2\n{\psi}^{1/2+6\alpha}\n{x}^{1/2+10\alpha}$$ and
$$
\n{\psi(x)-(h^2)^{\alpha}d(x)(h^2)^{\alpha}}\leq K\n{\psi}^{3/5}\e^{1/5}\n{x}.
$$
\end{enumerate}
\end{lemma}
\begin{proof}
According to Proposition~\ref{P_almostJordan}, for any $x\in \Aa$ we have 
$$
\n{\psi(x)^3-h\psi(x^2)\psi(x)}\leq\n{\psi(x)^2-h\psi(x^2)}\!\cdot\!\n{\psi(x)}\leq 108\n{\psi}\e\n{x}^3.
$$
Consider an arbitrary $x\in \Aa_\sa$ and define
$$
\omega(x)\coloneqq h\psi(x^2)\psi(x)\psi(x^2)h\in \oo{h\Bb h}_\sa.
$$
Obviously, we have
$$
\n{\psi(x)^3\psi(x^2)h-\omega(x)}\leq \n{\psi(x)^3-h\psi(x^2)\psi(x)}\!\cdot\!\n{\psi(x^2)h}\leq 108\n{\psi}^3\e\n{x}^5
$$
and
$$
\n{\psi(x)^3\psi(x^2)h-\psi(x)^5}\leq \n{\psi(x)^3}\!\cdot\!\n{(\psi(x)^2-h\psi(x))^\ast}\leq 108\n{\psi}^3\e\n{x}^5.
$$
Therefore, 
$$
\n{\psi(x)^5-\omega(x)}\leq 216\n{\psi}^3\e\n{x}^5.
$$
Now, an elementary verification shows that the~map $\R\ni t\longmapsto t^{1/5}$ is H\"older of order $\frac15$ (with constant $1$). Hence, appealing to the Aleksandrov--Peller Theorem~\ref{AP_thm} we get
\begin{equation}\label{psi_omega}
\n{\psi(x)-\omega(x)^{1/5}}\leq C\n{\psi}^{3/5}\e^{1/5}\n{x}\quad\mbox{for every }x\in \Aa_\sa,
\end{equation}
with an~absolute constant $C<\infty$. Of course, $\omega(x)^{1/5}\in\oo{h\Bb h}$, so the desired estimate \eqref{distance} has been proven for any self-adjoint $x\in \Aa$. For general $x\in A$ we simply consider its real and imaginary parts and we obtain \eqref{distance} with $K=2C$.

In order to prove the assertions (a)--(c), note that for every $x\in \Aa_\sa$ we have
$$
\psi(x^2)\psi(x)\psi(x^2)h^2\psi(x^2)\psi(x)\psi(x^2)\leq \n{\psi}^8\n{x}^{10}\!\cdot\! 1_{\Bb^\dag},
$$
whence
\begin{equation*}
\omega(x)^\ast\omega(x)=h\psi(x^2)\psi(x)\psi(x^2)h^2\psi(x^2)\psi(x)\psi(x^2)h\leq \n{\psi}^8\n{x}^{10} h^2.
\end{equation*}
By L\"owner's theorem (see \cite[Prop.~II.3.1.10]{blackadar}), the map $[0,\infty)\ni t\longmapsto t^{\beta}$ is operator monotone for each $\beta\in [0,1]$. Using this fact successively for $\beta=\frac12$ and $\beta=\frac25$ we obtain
\begin{equation}\label{omega_ineq}
(\omega(x)^{1/5})^\ast\omega(x)^{1/5}=\abs{\omega(x)}^{2/5}\leq \n{\psi}^{8/5}\n{x}^2(h^2)^{1/5}.
\end{equation}
Now, Proposition~\ref{polar}(i) applied to $\omega(x)^{1/5}$ and any $\alpha\in (0,\frac{1}{10})$ produces an~element $u(x)\in \Bb$ such that $\omega(x)^{1/5}=u(x)(h^2)^{\alpha}$ and 
\begin{equation}\label{better1}
\n{u(x)}\leq\n{\psi}^{8\alpha}\n{x}^{10\alpha}\n{(\n{\psi}^{8/5}\n{x}^2(h^2)^{1/5})^{1/2-5\alpha}}\leq \n{\psi}^{1-2\alpha}\n{x}.
\end{equation}
This, jointly with inequality \eqref{psi_omega}, proves the~assertion (a) for $x\in \Aa_\sa$ with constant $C$ instead of $K$. For an~arbitrary $x\in \Aa$ we take the~usual decomposition $x=x_1+\ii x_2$ with $x_1, x_2\in \Aa_\sa$, $\n{x_1},\n{x_2}\leq\n{x}$, and we define $u(x)=u(x_1)+\ii u(x_2)$. Then, obviously, $\n{u(x)}\leq 2\n{\psi}^{1-2\alpha}\n{x}$ and the~assertion (a) follows with $K=2C$.

The clause (b) is proved along the same lines by using Proposition~\ref{polar}(ii).

Finally, we apply Proposition~\ref{polar_twosided} to $\omega(x)^{1/5}$, for any $x\in \Aa_\sa$. In view of \eqref{omega_ineq} and the fact that $\omega(x)$ is self-adjoint, we infer that for each $\alpha\in (0,\frac{1}{20})$ there is $d(x)\in \Bb$ such that
$$
\omega(x)^{1/5}=(h^2)^{\alpha} d(x)(h^2)^{\alpha}
$$
and
$$
\n{d(x)}\leq \n{\psi}^{16\alpha}\n{x}^{20\alpha}\n{\omega(x)^{1/5}}^{1/2-10\alpha}.
$$
Since $\n{\omega(x)}\leq\n{\psi}^5\n{x}^5$, we have 
\begin{equation}\label{better2}
\n{d(x)}\leq \n{\psi}^{1/2+6\alpha}\n{x}^{1/2+10\alpha}.
\end{equation}
Again, appealing to inequality \eqref{psi_omega} and splitting any $x\in \Aa_\sa$ into its real and imaginary parts, we obtain the~assertion (c).
\end{proof}

\begin{corollary}\label{h_vs_psi}
The absolute constant $K$ from Lemma~{\rm \ref{close_to_hereditary}} has the following property. For any \cs-algebras $\Aa$ and $\Bb$ with $\Aa$ unital, and any self-adjoint $\eoz{\e}$ map $\psi\in\LL(\Aa,\Bb)$ with $h\coloneqq \psi(1_\Aa)$, at least one of the following two inequalities holds true:

\vspace*{2mm}\noindent
\begin{itemize}
\setlength{\itemsep}{5pt}
\item[{\rm (a)}] $\n{\psi}\leq\sqrt{(K+2)^5\e}$,
\item[{\rm (b)}] $\n{\psi}\leq (K+2)^5\n{h}$.
\end{itemize}
\end{corollary}
\begin{proof}
Take any $\alpha\in (0,\frac{1}{10})$. By Lemma~\ref{close_to_hereditary}(a), for each $x\in\Aa$ with $\n{x}\leq 1$ there exists $u(x)\in\Bb$ such that $\n{u(x)}\leq 2\n{\psi}^{1-2\alpha}$ and $\n{\psi(x)}\leq K\n{\psi}^{3/5}\e^{1/5}+2\n{\psi}^{1-2\alpha}\n{(h^2)^{\alpha}}$. Since $h$ is self-adjoint, we have $\n{(h^2)^\alpha}=\n{h}^{2\alpha}$. Passing to the limit as $\alpha\to \frac{1}{10}$ we thus obtain
\begin{equation}\label{psiKe}
\n{\psi}^{2/5}\leq K\e^{1/5}+2\n{\psi}^{1/5}\n{h}^{1/5}.
\end{equation}
For $h=0$ our assertion is trivial. If $h\neq 0$, we rewrite the last inequality as
\begin{equation}\label{psiKe2}
\Big(\frac{\n{\psi}}{\n{h}}\Big)^{1/5}\leq K\Big(\frac{\e}{\n{\psi}\n{h}}\Big)^{1/5}+2.
\end{equation}
Hence, if $\n{\psi}\n{h}\leq\e$, then \eqref{psiKe} yields inequality (a). Otherwise, \eqref{psiKe2} implies (b).
\end{proof}

It is hard to calculate the value of $K$ arising from the Aleksandrov--Peller inequality applied to the H\"older exponent $\alpha=\frac{1}{5}$. Birman, Koplienko and Solomyak \cite{BKS}, however, proved that for all positive self-adjoint operators $A$, $B$ on a~Hilbert space and for any $0<\alpha<1$, we have
$$
\n{A^\alpha-B^\alpha}\leq\n{A-B}^\alpha.
$$
Therefore, if in Lemma~\ref{close_to_hereditary} we additionally assume that $\psi$ is positive, we can apply the above inequality to the operators $\psi(x)^5$ and $\omega(x)$, for any fixed $x\in\Aa_+$. Hence, repeating our reasoning, we conclude that for every $x\in\Aa_+$ inequality \eqref{distance} and all assertions (a)--(c) hold true with constant $1$ instead of $K$, $\n{\psi}_+$ instead of $\n{\psi}$ and with estimates \eqref{better1} (both for $u(x)$ and $v(x)$) and \eqref{better2}. This leads to the~ following conclusion.
\begin{remark}
Suppose all the assumptions of Lemma~\ref{close_to_hereditary} are satisfied and that $\psi$ is positive. Then, by first considering positive elements and then using the Jordan decomposition, we infer that \eqref{distance} and assertions (a)--(c) hold true with constant $4$ instead of constants $K$ and $2$ and with $\n{\psi}_+$ in the place of $\n{\psi}$. Consequently, considering positive elements of $\Aa$ we obtain the following version of Corollary~\ref{h_vs_psi}.
\end{remark}
\begin{corollary}\label{h_vs_psi_pos}
For any \cs-algebras $\Aa$ and $\Bb$ with $\Aa$ unital, and any positive $\eoz{\e}$ map $\psi\in\LL(\Aa,\Bb)$ with $h\coloneqq \psi(1_\Aa)$, at least one of the following two inequalities holds true:

\vspace*{2mm}\noindent
\begin{itemize}
\setlength{\itemsep}{5pt}
\item[{\rm (a')}] $\n{\psi}_+\leq 166\sqrt{\e}$,
\item[{\rm (b')}] $\n{\psi}_+\leq 2\n{h}$.
\end{itemize}
\end{corollary}
\begin{proof}
As indicated in the remark above, we have $$\n{\psi}_+^{2/5}\leq \e^{1/5}+\n{\psi}_+^{1/5}\n{h}^{1/5}.$$ Take $R=13764$ and consider two cases: If $\n{\psi}_+\n{h}\leq R\e$, then $\n{\psi}_+\leq \sqrt{(1+R^{1/5})^5\e}$ which yields (a'). Otherwise $\n{\psi}_+\leq (1+R^{-1/5})^5\n{h}$ which implies (b').
\end{proof}

Recall that in the result by Alaminos, Extremera and Villena (Theorem~\ref{alaminos_thm}), an~important assumption was that the~given operator $T$ is surjective and the obtained estimate depended on the openness index of $T$. In Lemma~\ref{phi_P_phi} and Proposition~\ref{hZ_prop}, the obtained estimates depend on a~constant $M$ arising from the condition $\mathrm{C}^\ast(h)\subseteq \phi(\mathcal{Z}(\Aa))$. More precisely, any $M$ such that every element of the unit ball of $\mathrm{C}^\ast(h)$ can be approximated with arbitrarily small error by some $\phi(v)$ with $\n{v}\leq M$, does the job. Such a~constant was produced by inclusion \eqref{openn}, therefore, we define an~{\it openness index} of the restriction of $\phi$ to the center of $\Aa$ {\it relative to} $\mathrm{C}^\ast(h)$ by
\begin{equation}\label{open_rel}
\mathrm{op}_h(\phi\!\restriction_{\mathcal{Z}(\Aa)})\coloneqq \frac{1}{\sup\big\{r>0\colon rB_{\mathrm{C}^\ast(h)}\subset\oo{\phi(B_{\mathcal{Z}(\Aa)})}\big\}}.
\end{equation}
As we have already noted, if $\mathcal{Z}(\Aa)=\Aa$ and $\phi$ happens to be surjective, then
$$
\mathrm{op}_h(\phi\!\restriction_{\mathcal{Z}(\Aa)})\leq\mathrm{op}(\phi).
$$

\section{Decomposition---reducing to the unital case}
\noindent
We shall now collect our knowledge about approximate Jordan-like equations, almost commutation relations and ranges of $\eoz{\e}$ maps in order to prove a~decomposition result. Namely, we show that under certain conditions, a~self-adjoint $\eoz{\e}$ map $\psi$ can be decomposed into a~`small` part $\psi_\sss$ and a~`regular' part $\psi_\rrr$ which is a~unital $\eJh{\delta}$ map, with $\delta\to 0$ as $\e\to 0$. Therefore, the study of almost order zero maps may be sometimes reduced to the~study of almost Jordan $^\ast$-homomorphisms, and this will be the~topic of another paper.
\begin{lemma}\label{SUBS}
Let $S\in\LL(\Hh)$ be a~normal operator on a~Hilbert space $\Hh$ and $\mathsf{E}_0$ be its spectral measure. Let also $\delta>0$ and $T\in\LL(\Hh)$ be the~operator defined via functional calculus as $T=S\abs{S}^\delta$. Denote by $\mathsf{E}_1$ the~spectral measure of $T$. Then, for every $\e>0$, we have
$$
\mathsf{E}_0(\{\lambda\in\sigma(S)\colon \abs{\lambda}\leq\e\})=\mathsf{E}_1(\{\lambda\in\sigma(T)\colon \abs{\lambda}\leq\e^{1+\delta}\}).
$$
\end{lemma}
\begin{proof}
For any $\xi,\eta\in\Hh$ we define complex Borel measures $\mathsf{E}_{0,\xi,\eta}$ and $\mathsf{E}_{1,\xi,\eta}$ on $\sigma(S)$ and $\sigma(T)$, respectively, by the~formulas $\mathsf{E}_{0,\xi,\eta}=\langle \mathsf{E}_0(\cdot)\xi,\eta\rangle$ and $\mathsf{E}_{1,\xi,\eta}=\langle \mathsf{E}_1(\cdot)\xi,\eta\rangle$. By the~spectral theorem for $\mathrm{C}^\ast(S)$ and $\mathrm{C}^\ast(T)$, we have
\begin{equation}\label{subs1}
\int\limits_{\sigma(S)}f(\lambda\abs{\lambda}^{\delta})\,\mathrm{d}\mathsf{E}_{0,\xi,\eta}=\langle f(T)\xi,\eta\rangle=\int\limits_{\sigma(T)}f(\lambda)\,\mathrm{d}\mathsf{E}_{1,\xi,\eta}(\lambda)
\end{equation}
for all $\xi, \eta\in\Hh$ and $f\in C(\sigma(T))$. The~first integral can be transformed by substitution as follows. Consider the map
$$
\sigma(S)\ni\lambda\xmapsto[\phantom{xxx}]{}\Phi(\lambda)\coloneqq \lambda\abs{\lambda}^\delta\in\sigma(T).
$$
Then, $\Phi^{-1}(\sigma(T))=\sigma(S)$. Hence, for every $f\in C(\sigma(T))$, we have
\begin{equation}\label{subs2}
\int\limits_{\sigma(S)}f(\lambda\abs{\lambda}^{\delta})\,\mathrm{d}\mathsf{E}_{0,\xi,\eta}=\int\limits_{\sigma(T)} f(\lambda)\,\mathrm{d}\nu_{\xi,\eta},
\end{equation}
where $\nu_{\xi,\eta}$ is the~image measure of $\mathsf{E}_{0,\xi,\eta}$, that is, $\nu_{\xi,\eta}(A)=\mathsf{E}_{0,\xi,\eta}(\Phi^{-1}(A))$ for each Borel set $A\subseteq\sigma(T)$.

The right-hand sides of \eqref{subs1} and \eqref{subs2} are equal for each $f\in C(\sigma(T))$ and therefore $\mathsf{E}_{1,\xi,\eta}=\mathsf{E}_{0,\xi,\eta}\circ\Phi^{-1}$. Since $\xi,\eta\in\Hh$ were arbitrary, and since the~scalar measures $\mathsf{E}_{i,\xi,\eta}$ ($\xi,\eta\in\Hh$) uniquely determine the~spectral measure $\mathsf{E}_{i}$ ($i=1,2$), we have $\mathsf{E}_1=\mathsf{E}_0\circ\Phi^{-1}$. Hence,
\begin{equation*}
\begin{split}
\mathsf{E}_0(\{\lambda\in\sigma(S)\colon \abs{\lambda}\leq\e\}) &=\mathsf{E}_0(\{\lambda\in\sigma(S)\colon \abs{\lambda}^{1+\delta}\leq\e^{1+\delta}\})\\
&=\mathsf{E}_0\circ\Phi^{-1}(\{\lambda\in\sigma(T)\colon \abs{\lambda}\leq\e^{1+\delta}\})\\
&=\mathsf{E}_1(\{\lambda\in\sigma(T)\colon \abs{\lambda}\leq\e^{1+\delta}\}).\qedhere
\end{split}
\end{equation*}
\end{proof}

Below, $K<\infty$ stands for the absolute constant from Lemma~\ref{close_to_hereditary}. To avoid any irrelevant technical difficulties, we restrict ourselves to parameters $\e\in (0,1]$.
\begin{theorem}\label{L_decomposition}
Let $\Aa$, $\Bb$ be \cs-algebras with $\Aa$ unital and let $\pi$ be a~nondegenerate representation of $\Bb$ on a~Hilbert space $\Hh$. Let also $\psi\in\LL(\Aa,\Bb)$ be a~self-adjoint $\eoz{\e}$ map with some $\e\in (0,1]$ and with $h\coloneqq \psi(1_\Aa)$ which satisfies at least one of the following two conditions:
\begin{itemize}
\setlength{\itemsep}{4pt}
    \item[$(\mathsf{H}_1)$] $h$ is an algebraic element of $\Bb$;
    \item[$(\mathsf{H}_2)$] $\mathrm{C}^\ast(h)\subseteq \psi(\mathcal{Z}(\Aa))$.
\end{itemize}
Then, there exists a decomposition $\psi=\psi_\sss+\psi_\rrr$, where the operators $\psi_{\sss}, \psi_{\rrr}\in\LL(\Aa,\pi(\Bb)^\dprime)$ satisfy the~following conditions:

\vspace*{1mm}\noindent
\begin{itemize}
\setlength{\itemsep}{4pt}
\item[$(\mathsf{A}_1)$] 
$\displaystyle{
\n{\psi_\sss}\leq \Bigg\{\begin{array}{ll} (6K+7)\n{\psi}^{4/5}\e^{1/16} & \mbox{under }(\mathsf{H}_1)\\
(6K+7)\n{\psi}^{4/5}\e^{0.0003} & \mbox{under }(\mathsf{H}_2);
\end{array}}
$

\item[$(\mathsf{A}_2)$] $\psi_\rrr$ takes values in a~corner subalgebra $\mathcal{C}$ of $\oo{h\pi(\Bb)^\dprime h}$;

\item[$(\mathsf{A}_3)$] either $\psi_\rrr=0$ or $\psi_\rrr(1_\Aa)$ is invertible in $\mathcal{C}$, in which case $\psi_\rrr(1_\Aa)^{-1}\psi_\rrr(\,\cdot\,)$ is a~unital $\eJh{\delta}$ map with
$$
\hspace*{-92pt}
\delta=\Bigg\{\begin{array}{ll}
24\big(C^2(K+2)^5+10C+17\big)\n{\psi}\e^{1/16} & \mbox{under }(\mathsf{H}_1)\\
D\big(\n{\psi}^{391/128}+\n{\psi}^{31/8}\big)\mathcal{O}(\e^{0.0003}) & \mbox{under }(\mathsf{H}_2),
\end{array}
$$
where $C$ depends only on the degree of algebraicity of $h$ as in Proposition~{\rm \ref{alg_prop}}, whereas $D$ depends only on $\mathrm{op}_h(\psi\!\restriction_{\mathcal{Z}(\Aa)})$.
\end{itemize}
Moreover, if $\psi$ is assumed to be positive, then under hypothesis $(\mathsf{H}_1)$ we have $$\n{\psi_\sss}\leq 37\n{\psi}^{4/5}\e^{1/16}\quad\mbox{ and }\quad \delta=24(2C^2+10C+17)\n{\psi}\e^{1/16}.
$$
\end{theorem}
\begin{proof}
First, observe that in the case $h=0$ our assertion follows easily from Proposition~\ref{P_almostJordan}. Indeed, for every $x\in \Aa$ we have $\n{\psi(x)^2}\leq 108\e\n{x}^2$ and since $\psi$ is self-adjoint, we have $\n{\psi(y)}\leq \sqrt{108\e}\n{y}$ for each $y\in \Aa_\sa$. It follows that $\n{\psi(x)}\leq 2\sqrt{108\e}\n{x}$ for $x\in \Aa$. So, in this case we simply set $\psi_\sss=\psi$ and $\psi_\rrr=0$. In the~rest of the~proof we thus assume that $h\not=0$. 

Secondly, observe that our assertion is also valid with $\psi_\sss=\psi$ and $\psi_\rrr=0$ in the case where inequality (a) from Corollary~\ref{h_vs_psi} holds true. Henceforth, we can thus assume that $\psi$ satisfies the other inequality:
\begin{equation}\label{norm_of_psi}
    \n{\psi}\leq (K+2)^5\n{h}.    
\end{equation}
Consequently, there is nothing to prove if $\n{h}\leq \e^{1/2}$ (once more, direct verification shows that the assertion holds true with $\psi_\rrr=0$), so we can additionally assume from this point on that $\n{h}>\e^{1/2}$. For the sake of readability, we will divide the rest of the proof into several steps.

\vspace*{3mm}\noindent
{\sc Step 1: Corner decomposition with parameter} 

\vspace*{1mm}\noindent
We {\it claim} that for any $\theta>0$ there is a~sufficiently small $\beta>0$ such that 
\begin{equation}\label{psi_hpsih}
\n{\psi(x)-(h^2)^\beta\psi(x)(h^2)^\beta}\leq (3K\n{\psi}^{3/5}\e^{1/5}+\theta)\n{x}\quad\mbox{for every }x\in \Aa. 
\end{equation} 
Indeed, for arbitrary $x\in \Aa$ and $0<\alpha<\frac{1}{20}$, let $d(x)$ be the~element of $\Bb$ produced by Lemma~\ref{close_to_hereditary}(c). Then, for any $\beta>0$, we have
\begin{equation*}
\begin{split}
\n{(h^2)^\beta\psi(x)(h^2)^\beta- &(h^2)^{\alpha+\beta}d(x)(h^2)^{\alpha+\beta}}\\
& \leq\n{\psi}^{4\beta}\!\cdot\!\n{\psi(x)-(h^2)^{\alpha}d(x)(h^2)^{\alpha}}\leq K\n{\psi}^{3/5+4\beta}\e^{1/5}\n{x}.
\end{split}
\end{equation*}
Therefore,
\begin{equation}\label{nu1}
\begin{split}
\n{\psi(x)- &(h^2)^{\beta}\psi(x)(h^2)^{\beta}}\\
&\leq \n{(h^2)^{\alpha}d(x)(h^2)^{\alpha}-(h^2)^{\alpha+\beta}d(x)(h^2)^{\alpha+\beta}}+K\n{\psi}^{3/5}(1+\n{\psi}^{4\beta})\e^{1/5}\n{x}.
\end{split}
\end{equation}
Define $\chi_{\alpha,\beta}=\n{(h^2)^\alpha-(h^2)^{\alpha+\beta}}$ and estimate the~first summand in \eqref{nu1} as follows:
\begin{equation*}
\begin{split}
\n{(h^2)^{\alpha}d(x)(h^2)^{\alpha}-&(h^2)^{\alpha+\beta}d(x)(h^2)^{\alpha+\beta}}\\
&\leq \n{(h^2)^\alpha}\!\cdot\!\n{d(x)(h^2)^\alpha-d(x)(h^2)^{\alpha+\beta}}+\n{d(x)(h^2)^{\alpha+\beta}}\!\cdot\!\chi_{\alpha,\beta}\\
&\leq \n{\psi}^{2\alpha}(1+\n{\psi}^{2\beta})\n{d(x)}\!\cdot\!\chi_{\alpha,\beta}\\
&\leq 2\n{\psi}^{1/2+8\alpha}(1+\n{\psi}^{2\beta})\n{x}^{1/2+10\alpha}\!\cdot\!\chi_{\alpha,\beta}.
\end{split}
\end{equation*}
By the Gelfand--Naimark theorem for \cs$(h)$, it is easily seen that $\lim_{\beta\to 0+}\chi_{\alpha,\beta}=0$ uniformly for $\alpha$'s from any bounded set, in particular, for $\alpha\in (0,\frac{1}{20})$. Consequently, for any fixed $\theta>0$, one can pick $\beta>0$ so that the right-hand side of the~last inequality is smaller than $\theta\n{x}^{1/2+10\alpha}$, for each $\alpha\in (0,\frac{1}{20})$. Of course, we can also assume that $\n{\psi}^{4\beta}<2$. Passing to the~limit as $\alpha\to\frac{1}{20}$ from the~left and using \eqref{nu1} we obtain the~announced inequality \eqref{psi_hpsih}.

Fix any $\theta>0$. Having established the claim, we define an~operator $\wpsi\colon \Aa\to \oo{h\Bb h}$ by
$$
\wpsi(x)=(h^2)^\beta\psi(x)(h^2)^\beta,
$$
where $\beta$ is chosen to be any positive number for which inequality \eqref{psi_hpsih} holds true. Observe that, in fact, the~so-defined operator takes values in the~hereditary \cs-subalgebra $\oo{h\Bb h}$. Indeed,
$$
(h^2)^\beta\in \mathrm{C}^\ast((h^2)^3)\subseteq\oo{h^2 \Bb h^2}\subseteq\oo{h\Bb h}
$$
and therefore $(h^2)^\beta$ can be approximated in norm by elements from $h\Bb h$, so the~same is true for $\wpsi(x)$, as multiplication is jointly norm continuous.

Henceforth, for simplicity of notation, we assume that $\Bb $ acts nondegenerately on $\Hh$, that is, $\pi$ is the~identity (however, we still use the~symbol $\pi(\Bb)^\dprime$ rather than $\Bb^\dprime$ to avoid any confusion). We will use the following convention: We write $a\lesssim b$ provided that for any $\eta>0$, the inequality $a\leq b+\eta$ is true after possibly decreasing the parameter $\beta$, while not violating any other statements along the proof. Let us also point out that working under assumption ($\mathsf{H}_1$) we keep track on actual constants of approximation, whereas under ($\mathsf{H}_2$) we just care about whether they depend only on the openness index relative to $\mathrm{C}^\ast(h)$. 

Define
$$
g\coloneqq\wpsi(1_\Aa)=h\!\cdot\! \abs{h}^{4\beta}\in \Bb_\sa;
$$
note that $g\not=0$ because $h\not=0$. Let $\mathsf{E}_0$ and $\mathsf{E}_1$ be the~spectral measures of $h$ and $g$, respectively. For any parameter $\gamma>0$ satisfying $\e^\gamma<\n{h}$ (equivalently: $\e^{\gamma(1+4\beta)}<\n{g}$), we consider the~spectral projection in $\pi(\Bb)^\dprime$ given by
\begin{equation}\label{p_gamma_1}
p_\gamma\coloneqq\mathsf{E}_1(\{\lambda\in\sigma(g)\colon \abs{\lambda}\leq \e^{\gamma(1+4\beta)}\}).
\end{equation}
In fact, $p_\gamma$ does not depend on $\beta$, since from Lemma~\ref{SUBS} it follows that
\begin{equation}\label{p_gamma_0}
p_\gamma=\mathsf{E}_0(\{\lambda\in\sigma(h)\colon \abs{\lambda}\leq\e^\gamma\}).
\end{equation}
Therefore, for every $\delta>0$, we have
\begin{equation*}\label{HP}
p_\gamma (h^2)^\delta=(h^2)^\delta p_\gamma=\!\!\!\int\limits_{[-\e^\gamma,\e^\gamma]\cap\sigma(h)}\!\!\!\!\abs{\lambda}^{2\delta}\,\mathrm{d}\mathsf{E}_0(\lambda)
\end{equation*}
and hence
\begin{equation}\label{HP_norm}
\n{p_\gamma (h^2)^\delta}=\n{(h^2)^\delta p_\gamma}=\,\sup_{[-\e^\gamma,\e^\gamma]\cap\sigma(h)}\,\abs{\lambda}^{2\delta}\leq\e^{2\gamma\delta}.
\end{equation}
Similarly,
\begin{equation}\label{HI-P_norm}
\n{(1_\Hh-p_\gamma) (h^2)^\delta} =\n{(h^2)^\delta (1_\Hh-p_\gamma)}\leq\,\sup_{\sigma(h)}\,\abs{\lambda}^{2\delta}\leq\n{h}^{2\delta}\leq\n{\psi}^{2\delta}.
\end{equation}

Consider the corner decomposition
$$
\wpsi(x)=p_\gamma\wpsi(x)p_\gamma+ p_\gamma\wpsi(x) (1_{\Hh}-p_\gamma)+(1_{\Hh}-p_\gamma)\wpsi(x)p_\gamma+(1_{\Hh}-p_\gamma)\wpsi(x)(1_{\Hh}-p_\gamma).
$$
For each $\gamma>0$ with $\e^\gamma<\n{h}$, we may thus write $\psi=\psi_{\sss,\gamma}+\psi_{\rrr,\gamma}$, where
\begin{equation}\label{psi_s_def}
\psi_{\sss,\gamma}(x)=(\psi(x)-\wpsi(x))+p_\gamma\wpsi(x)p_\gamma+ p_\gamma\wpsi(x) (1_{\Hh}-p_\gamma)+(1_{\Hh}-p_\gamma)\wpsi(x)p_\gamma
\end{equation}
and
\begin{equation}\label{psi_r_def}
\psi_{\rrr,\gamma}(x)=(1_{\Hh}-p_\gamma)\wpsi(x)(1_{\Hh}-p_\gamma).
\end{equation}
In what follows, we shall prove that regardless of which of the hypotheses ($\mathsf{H}_1$) and ($\mathsf{H}_2$) holds true, there is an~appropriate value of $\gamma$ for which the~operators $\psi_\sss=\psi_{\sss,\gamma}$ and $\psi_\rrr=\psi_{\rrr,\gamma}$ enjoy all the~desired properties.

\vspace*{3mm}\noindent
{\sc Step 2: Upper bound for $\n{\psi_{\sss,\gamma}}$}

\vspace*{1mm}\noindent
First, in order to estimate the~norm of $p_\gamma\wpsi p_\gamma$, we apply Lemma~\ref{close_to_hereditary}(a). For any $0<\alpha<\frac{1}{10}$ and $x\in \Aa$, let $u(x)\in \Bb$ be as in that lemma. Notice that in view of inequality~\eqref{HP_norm} and the~definition of $\wpsi$, we have
\begin{equation*}
\begin{split}
\n{p_\gamma\wpsi(x)p_\gamma-p_\gamma (h^2)^\beta & u(x)(h^2)^{\alpha+\beta}p_\gamma}\\
&\leq \n{p_\gamma (h^2)^\beta}\!\cdot\!\n{\psi(x)-u(x)(h^2)^\alpha}\!\cdot\! \n{(h^2)^\beta p_\gamma}\\
&\leq K\n{\psi}^{3/5}\e^{1/5+4\beta\gamma}\n{x}.
\end{split}
\end{equation*}
The estimate for $\n{u(x)}$ given in Lemma~\ref{close_to_hereditary}(a) jointly with inequality~\eqref{HP_norm} yields
\begin{equation*}
\n{p_\gamma (h^2)^\beta u(x)(h^2)^{\alpha+\beta} p_\gamma}\leq 2\n{\psi}^{1-2\alpha}\e^{2(\alpha+2\beta)\gamma}\n{x}.
\end{equation*}
Hence, for every $0<\alpha<\frac{1}{10}$, we have
\begin{equation}\label{PGP}
\n{p_\gamma\wpsi p_\gamma}\leq K\n{\psi}^{3/5}\e^{1/5+4\beta\gamma}+2\n{\psi}^{1-2\alpha}\e^{2(\alpha+2\beta)\gamma}.
\end{equation}

Now, in order to estimate the norm of $p_\gamma\wpsi(1_\Hh-p_\gamma)$, we appeal to Lemma~\ref{close_to_hereditary}(b). For any $0<\alpha<\frac{1}{10}$ and $x\in \Aa$, let $v(x)\in \Bb$ be as in that lemma. By inequalities~\eqref{HP_norm} and \eqref{HI-P_norm}, we have
\begin{equation*}
\begin{split}
\n{p_\gamma\wpsi(x)(1_\Hh-p_\gamma)- p_\gamma (h^2)^{\alpha+\beta} &v(x)(h^2)^{\beta}(1_\Hh-p_\gamma)}\\
&\leq\n{p_\gamma (h^2)^\beta}\!\cdot\!\n{\psi(x)-(h^2)^\alpha v(x)}\!\cdot\!\n{(h^2)^\beta(1_\Hh-p_\gamma)}\\
&\leq K\n{\psi}^{3/5+2\beta}\e^{1/5+2\beta\gamma}\n{x}.
\end{split}
\end{equation*}
Using the estimate for $\n{v(x)}$, along with \eqref{HP_norm} and \eqref{HI-P_norm}, we obtain
$$
\n{p_\gamma (h^2)^{\alpha+\beta}v(x)(h^2)^{\beta}(1_\Hh-p_\gamma)}\leq 2\n{\psi}^{1-2\alpha+2\beta}\e^{2(\alpha+\beta)\gamma}\n{x}.
$$
Hence,
\begin{equation}\label{PGI-P}
\n{p_\gamma\wpsi(1_\Hh-p_\gamma)}\leq K\n{\psi}^{3/5+2\beta}\e^{1/5+2\beta\gamma}+2\n{\psi}^{1-2\alpha+2\beta}\e^{2(\alpha+\beta)\gamma}.
\end{equation}
By a similar argument, and appealing to the~``left-handed" assertion (a) from Lemma~\ref{close_to_hereditary} instead of the~``right-handed" assertion (b), we obtain
\begin{equation}\label{I-PGP}
\begin{split}
\n{(1_{\Hh}-p_\gamma)\wpsi p_\gamma}\leq K\n{\psi}^{3/5+2\beta}\e^{1/5+2\beta\gamma}+2\n{\psi}^{1-2\alpha+2\beta}\e^{2(\alpha+\beta)\gamma}.
\end{split}
\end{equation}
Combining inequalities \eqref{psi_hpsih}, \eqref{PGP}, \eqref{PGI-P} and \eqref{I-PGP}, and recalling formula~\eqref{psi_s_def}, we obtain
\begin{equation*}
\begin{split}
\n{\psi_{\sss,\gamma}} &\leq 3K\n{\psi}^{3/5}\e^{1/5}+\theta\\
&\quad+K\n{\psi}^{3/5}\e^{1/5+4\beta\gamma}+2\n{\psi}^{1-2\alpha}\e^{2(\alpha+2\beta)\gamma}\\
&\quad+2K\n{\psi}^{3/5+2\beta}\e^{1/5+2\beta\gamma}+4\n{\psi}^{1-2\alpha+2\beta}\e^{2(\alpha+\beta)\gamma}.
\end{split}
\end{equation*}
Taking a suitable $\theta>0$ (depending on $\e$, $\gamma$ and $\n{\psi}$), picking $\beta>0$ small enough and passing to the limit as $\alpha\to\frac{1}{10}-$, we can guarantee that
\begin{equation}\label{norm_psi_s}
\n{\psi_{\sss,\gamma}}\leq 6K\n{\psi}^{3/5}\e^{1/5}+7\n{\psi}^{4/5}\e^{\gamma/5}.
\end{equation}

\vspace*{3mm}\noindent{\sc Step 3: Properties of the regular part}

\vspace*{1mm}\noindent
Now, we {\it claim} that $\psi_{\rrr,\gamma}$ has the following properties:

\vspace*{1mm}\noindent
\begin{enumerate}[{(i)}]
\setlength{\itemindent}{-12pt}
\setlength{\itemsep}{4pt}
\setlength{\labelsep}{6pt}
\item $\psi_{\rrr,\gamma}\colon \Aa\longrightarrow \mathcal{C}_\gamma\coloneqq\oo{(1_\Hh-p_\gamma)h\pi(\Bb)^\dprime h(1_\Hh-p_\gamma)}$,
\item $\n{\psi_{\rrr,\gamma}}\leq\n{\psi}^{1+4\beta}$,
\item $\psi_{\rrr,\gamma}$ is self-adjoint,
\item $\psi_{\rrr,\gamma}(1_\Aa)$ is invertible in $\mathcal{C}_\gamma$ and $\n{\psi_{\rrr,\gamma}(1_\Aa)^{-1}}\leq \e^{-\gamma(1+4\beta)}$.
\setlength{\itemindent}{3pt}
\end{enumerate}

\vspace*{1mm}\noindent
The property (i) follows from formula \eqref{psi_r_def} and the~fact that $\wpsi$ takes values in $\oo{h\Bb h}$. The~clauses (ii) and (iii) follow from the~very definition. To see that (iv) holds true, observe that $\psi_{\rrr,\gamma}(1_\Aa)=(1_\Hh-p_\gamma)g$ and that, in view of \eqref{p_gamma_1}, we have
$$
\sigma((1_\Hh-p_\gamma)g)\subseteq [-\n{g},-\e^{\gamma(1+4\beta)}]\cup [\e^{\gamma(1+4\beta)},\n{g}],
$$
the spectrum taken in the corner algebra $\mathcal{C}_\gamma$.

We shall now verify that $\psi_{\rrr,\gamma}$ is an~almost order zero map. To this end, first observe that in view of formula \eqref{p_gamma_0}, we have
$$
(1_\Hh-p_\gamma)(1_\Hh-(h^2)^{2\beta})=\!\!\!\!\!\int\limits_{\{\lambda\in\sigma(h)\colon \abs{\lambda}>\e^\gamma\}}\!\!\!\!\!\!(1-\abs{\lambda}^{4\beta})\,\mathrm{d}\mathsf{E}(\lambda).
$$
Therefore,
\begin{equation}\label{kappa_to_zero}
\kappa_\beta\coloneqq\n{(1_\Hh-p_\gamma)(1_\Hh-(h^2)^{2\beta})}\leq\sup_{\e^\gamma<t\leq\n{h}}\vert 1-t^{4\beta}\vert\xrightarrow[\,\beta\to 0+\,]{}0.
\end{equation}
Fix any $x,y\in A_+$ with $xy=0$. By the definition \eqref{psi_r_def} of $\psi_{\rrr,\gamma}$, and the fact that $p_\gamma$ and $h$ commute, we have
\begin{equation}\label{psirg_formula}
\psi_{\rrr,\gamma}(x)\psi_{\rrr,\gamma}(y)=(1_\Hh-p_\gamma)(h^2)^\beta\psi(x)(1_\Hh-p_\gamma)(h^2)^{2\beta}\psi(y)(1_\Hh-p_\gamma)(h^2)^\beta.
\end{equation}
Plainly, we have
\begin{equation*}\label{eozPR}
\n{\psi(x)(1_\Hh-p_\gamma)(h^2)^{2\beta}-\psi(x)(1_\Hh-p_\gamma)}\leq \kappa_\beta\n{\psi}\n{x}
\end{equation*}
and hence, in view of \eqref{psirg_formula}, we get
\begin{equation}\label{psi_rge}
\n{\psi_{\rrr,\gamma}(x)\psi_{\rrr,\gamma}(y)}\leq \n{(1_\Hh-p_\gamma)(h^2)^\beta}^2\Big(\kappa_\beta\n{\psi}^2\n{x}\n{y}+\n{\psi(x)(1_\Hh-p_\gamma)\psi(y)}\Big).
\end{equation}
Recalling that $\n{h}>\e^{\gamma}$ and that we have assumed \eqref{norm_of_psi}, we infer that under hypothesis ($\mathsf{H}_1$), Proposition~\ref{alg_prop} yields 
\begin{equation*}
    \n{P(h)\psi(x)-\psi(x)P(h)}\leq C\e^{1-\gamma}\sup_{\abs{z}=\n{h}}\abs{P(z)}\cdot\n{x} 
\end{equation*}
for every $P\in\C[z]$ with $P(0)=0$, where $C$ depends only on the degree of algebraicity of $h$. On the other hand, it follows from Proposition~\ref{hZ_prop} that under hypothesis ($\mathsf{H}_2$), we have
\begin{equation*}
\begin{split}
    \n{P(h)\psi &(x)-\psi(x)P(h)}\\
    &\leq 8(K+2)^5\mathrm{op}_h(\psi\!\restriction_{\mathcal{Z}(\Aa)})\big(\n{\psi}^{15/8}\mathcal{O}(\e^{1/16})+24\e\big)\sup_{\abs{z}=\n{h}}\abs{P(z)}\!\cdot\!\n{x},
    \end{split}
\end{equation*}
where in the role of $M$ we took the openness index relative to $\mathrm{C}^\ast(h)$ and, once again, we used inequality \eqref{norm_of_psi}. Note also that we have omitted the term $2\e^{1/2}$ under the maximum sign, as in our case $\n{\psi}\geq\n{h}>\e^{1/2}$ (see the beginning of the proof of Proposition~\ref{hZ_prop}). Since $\n{\psi}^{15/8}>\e^{15/16}$, the term $24\e$ is majorized by $\n{\psi}^{15/8}\mathcal{O}(\e^{1/16})$ and we can rewrite the above inequality in a~simpler form:
\begin{equation*}
    \n{P(h)\psi(x)-\psi(x)P(h)}\leq D\n{\psi}^{15/8}\mathcal{O}(\e^{1/16})\sup_{\abs{z}=\n{h}}\abs{P(z)}\cdot\n{x},
\end{equation*}
where $D$ depends only on the openness index of $\psi\!\restriction_{\mathcal{Z}(\Aa)}$ relative to $\mathrm{C}^\ast(h)$.

In each case, we can apply Lemma~\ref{sp_proj} to the spectral projection $V=1_\Hh-p_\gamma$ and the operator $T=\psi(x)$. In this way, we obtain an~estimate on the norm of the commutator:
\begin{equation}\label{V_1-p}
    \n{[1_\Hh-p_\gamma,\psi(x)]}\leq\left\{\begin{array}{ll}
    C\e^{1-\gamma}\n{x} & \mbox{under }(\mathsf{H}_1)\\
    D\n{\psi}^{15/8}\mathcal{O}(\e^{1/16})\n{x} & \mbox{under }(\mathsf{H}_2).
    \end{array}\right.
\end{equation}
Therefore, by the fact that $\psi$ is $\eoz{\e}$, we obtain
\begin{equation}\label{H1H2_ineq}
    \begin{split}
        \n{\psi(x)(1_\Hh-p_\gamma)\psi(y)} &\leq \n{(1_\Hh-p_\gamma)\psi(x)\psi(y)}+\n{[1_\Hh-p_\gamma,\psi(x)]}\n{\psi(y)}\\
        &\leq \n{x}\n{y}\cdot\left\{\begin{array}{ll}
    C\n{\psi}\e^{1-\gamma}+\e & \mbox{under }(\mathsf{H}_1)\\
    D\n{\psi}^{23/8}\mathcal{O}(\e^{1/16})+\e & \mbox{under }(\mathsf{H}_2).
    \end{array}\right.
    \end{split}
\end{equation}
In view of \eqref{kappa_to_zero}, by decreasing $\beta$ if necessary, we can assure that $\kappa_\beta$ is arbitrarily small and $\n{\psi}^{4\beta}$ close to $1$. Hence, combining \eqref{psi_rge} and \eqref{H1H2_ineq}, we obtain an~estimate on $\n{\psi_{\rrr,\gamma}(x)\psi_{\rrr,\gamma}(y)}$ which means that the operator $\psi_{\rrr,\gamma}$ is $\eoz{\xi}$, where 
\begin{equation}\label{xioz}
        \xi=\left\{\begin{array}{ll}
    C\n{\psi}\e^{1-\gamma}+2\e & \mbox{under }(\mathsf{H}_1)\\
    D\n{\psi}^{23/8}\mathcal{O}(\e^{1/16}) & \mbox{under }(\mathsf{H}_2).
    \end{array}\right.
\end{equation}
Note that in the second row there is a~new constant $D$ which still depends only on $\mathrm{op}_h(\psi\!\restriction_{\mathcal{Z}(\Aa)})$, while the value of $C$ did not change and it comes from Proposition~\ref{alg_prop} applied to $N$ being the degree of algebraicity of $h$ and $M=(K+2)^5$.

\vspace*{3mm}\noindent
{\sc Step 4: Estimating the norm of the commutator $[\psi_{\rrr,\gamma}(1_\Aa)^{-1}, \psi_{\rrr,\gamma}(x)]$}

\vspace*{1mm}\noindent
Fix any $x\in\Aa_+$ and denote $h_\gamma\coloneqq \psi_{\rrr,\gamma}(1_\Aa)$. Since $\e^\gamma<\n{h}$, the projection $1_\Hh-p_\gamma$ does not affect the largest in absolute value elements of $\sigma(h)$ and hence $\n{h_\gamma}=\n{h}^{1+4\beta}$. By our assumption \eqref{norm_of_psi} and inequality (ii), we thus obtain  
\begin{equation}\label{norm_of_psir}
\n{\psi_{\rrr,\gamma}}\leq (K+2)^{5(1+4\beta)}\n{h_\gamma}.
\end{equation}

First, assume ($\mathsf{H}_1$) holds true. Notice that the spectrum of $h_\gamma$ is finite and has no more elements than $\sigma(h)$. Hence, $h_\gamma$ is also algebraic of degree not larger than the degree of $h$. By virtue of Proposition~\ref{alg_prop}, for every $P\in\C[z]$ with $P(0)=0$ we have
\begin{equation}\label{est_comm1}
    \n{P(h_\gamma)\psi_{\rrr,\gamma}(x)-\psi_{\rrr,\gamma}(x)P(h_\gamma)}\lesssim\frac{C\xi}{\n{h_\gamma}}\sup_{\abs{z}=\n{h_\gamma}}\abs{P(z)}\cdot\n{x}.
\end{equation}
By \eqref{norm_of_psi}, \eqref{xioz} and the fact that $\n{h}>\e^{\gamma}$, we have
\begin{equation*}
        \frac{\xi}{\n{h_\gamma}}=\frac{C\n{\psi}\e^{1-\gamma}+2\e}{\n{h}^{1+4\beta}}\leq C(K+2)^5\frac{\e^{1-\gamma}}{\n{h}^{4\beta}}+\frac{2\e}{\e^{\gamma(1+4\beta)}}\lesssim (C(K+2)^5+2)\e^{1-\gamma}.
\end{equation*}
Therefore, \eqref{est_comm1} yields
\begin{equation}\label{est_comm2}
     \n{P(h_\gamma)\psi_{\rrr,\gamma}(x)-\psi_{\rrr,\gamma}(x)P(h_\gamma)}\lesssim (C^2(K+2)^5+2C)\e^{1-\gamma}\sup_{\abs{z}=\n{h_\gamma}}\abs{P(z)}\cdot\n{x}.
\end{equation}
Consider a linear operator defined on the space of complex continuous functions on the ring $\{\e^{\gamma(1+4\beta)}\leq\abs{z}\leq\n{h_\gamma}\}\supseteq \sigma(h_\gamma)$ by the formula $f\mapsto f(h_\gamma)\psi_{\rrr,\gamma}(x)-\psi_{\rrr,\gamma}(x)f(h_\gamma)$. The Stone--Weierstrass theorem and estimate \eqref{est_comm2} imply that, by decreasing $\beta$, the norm of such an~operator can be estimated by any number larger than $(C^2(K+2)^5+2C)\e^{1-\gamma}$. Taking the inverse function $z\mapsto z^{-1}$, whose supremum norm on the ring equals $\e^{-\gamma(1+4\beta)}$, we obtain
\begin{equation}\label{est_comm3}
    \n{h_\gamma^{-1}\psi_{\rrr,\gamma}(x)-\psi_{\rrr,\gamma}(x)h_\gamma^{-1}}\lesssim (C^2(K+2)^5+2C)\e^{1-2\gamma}\n{x}.
\end{equation}

Now, assume ($\mathsf{H}_2$) holds true and notice that 
$$
\mathrm{C}^\ast(h)=(h^2)^\beta \mathrm{C}^\ast(h) (h^2)^\beta\subseteq (h^2)^\beta \psi(\mathcal{Z}(\Aa))(h^2)^\beta=\wpsi(\mathcal{Z}(\Aa)).
$$
(For the first equality one can use Proposition~\ref{polar_twosided}, or simpler, apply the Gelfand--Naimark theorem to $\mathrm{C}^\ast(h)$.) Consequently, $\mathrm{C}^\ast(h_\gamma)\subseteq (1_\Hh-p_\gamma)\mathrm{C}^\ast(h)\subseteq \psi_{\rrr,\gamma}(\mathcal{Z}(\Aa))$, which means that $\psi_{\rrr,\gamma}$ satisfies the condition analogous to ($\mathsf{H}_2$). More precisely, any $w\in \mathrm{C}^\ast(h_\gamma)$ can be written as $w=(1_\Hh-p_\gamma)(h^2)^{2\beta}w^\prime$ with $w^\prime\in\mathrm{C}^\ast(h)$ satisfying $\n{w^\prime}=\n{h}^{-4\beta}\n{w}$. Hence, appealing to formula~\eqref{open_rel}, we easily get that
$$
\mathrm{op}_{h_\gamma}(\psi_{\rrr,\gamma}\!\restriction_{\mathcal{Z}(\Aa)})\lesssim \mathrm{op}_h(\psi\!\restriction_{\mathcal{Z}(\Aa)}).
$$
By Proposition~\ref{hZ_prop} and inequalities \eqref{norm_of_psir}, (ii), for every $P\in\C[z]$ with $P(0)=0$, we have
\begin{equation*}
\begin{split}
\n{P(h_\gamma)\psi_{\rrr,\gamma}(x) &-\psi_{\rrr,\gamma}(x)P(h_\gamma)}\\
&\lesssim 
8(K+2)^5 \mathrm{op}_h(\psi\!\restriction_{\mathcal{Z}(\Aa)})\big(\n{\psi}^{15/8}\mathcal{O}(\xi^{1/16})+24\xi\big)\sup_{\abs{z}=\n{h_\gamma}}\abs{P(z)}\cdot\n{x}.
\end{split}
\end{equation*}
Hence, in view of \eqref{xioz},
\begin{equation*}
\begin{split}
\n{P(h_\gamma)\psi_{\rrr,\gamma}(x) &-\psi_{\rrr,\gamma}(x)P(h_\gamma)}\\
&\leq D\Big\{\n{\psi}^{263/128}\mathcal{O}(\e^{1/256})+\n{\psi}^{23/8}\mathcal{O}(\e^{1/16})\Big\}\sup_{\abs{z}=\n{h_\gamma}}\abs{P(z)}\cdot\n{x},
\end{split}
\end{equation*}
where $D$ is a new constant depending only on $\mathrm{op}_h(\psi\!\restriction_{\mathcal{Z}(\Aa)})$. Arguing as before we conclude that, for a~possibly new value of $D$, we have
\begin{equation}\label{est_comm4}
    \n{h_\gamma^{-1}\psi_{\rrr,\gamma}(x)-\psi_{\rrr,\gamma}(x)h_\gamma^{-1}} \leq D\e^{-\gamma}\Big\{\n{\psi}^{263/128}\mathcal{O}(\e^{1/256})+\n{\psi}^{23/8}\mathcal{O}(\e^{1/16})\Big\}\cdot\n{x}.
\end{equation}

\vspace*{3mm}\noindent{\sc Step 5: Picking the right parameter}

\vspace*{1mm}\noindent
We keep $x\in\Aa_+$ fixed. As we have proved that $\psi_{\rrr,\gamma}$ is $\eoz{\xi}$, it follows from Proposition~\ref{P_almostJordan} (see the beginning of the proof) that $\n{\psi_{\rrr,\gamma}(x)^2-h_\gamma \psi_{\rrr,\gamma}(x^2)}\leq 8\xi\n{x}^2$. Taking the inverse of $h_\gamma^2$ in the corner algebra $\mathcal{C}_\gamma$ we get
\begin{equation}\label{Jordan_s5_1}
    \n{h_\gamma^{-2}\psi_{\rrr,\gamma}(x)^2-h_\gamma^{-1} \psi_{\rrr,\gamma}(x^2)}\leq 8\xi\n{h_\gamma^{-1}}^2\n{x}^2\lesssim 8\xi\e^{-2\gamma}\n{x}^2.
\end{equation}
Define an operator $\Xi_\gamma\colon\Aa\to\mathcal{C}_\gamma$ by $\Xi_\gamma(x)=h_\gamma^{-1}\psi_{\rrr,\gamma}(x)$. Note that $\Xi_\gamma$ is unital and since by (iii) we have
\begin{equation}\label{delta_sa_Jordan}
\n{\,\Xi_\gamma(y^\ast)-\Xi_\gamma(y)^\ast}=\n{h_\gamma^{-1}\psi_{\rrr,\gamma}(y)-\psi_{\rrr,\gamma}(y)h_\gamma^{-1}}\qquad (y\in\Aa),
\end{equation}
it is also $\esa{\delta}$, where $\delta$ can be estimated with the aid of either \eqref{est_comm3} or \eqref{est_comm4} depending on whether ($\mathsf{H}_1$) or ($\mathsf{H}_2$) holds true (here, we use the Jordan decomposition for the real and imaginary parts of $y$ and hence the error terms given by \eqref{est_comm3} and \eqref{est_comm4} should be multiplied by $4$). 

By inequalities (ii), (iv), \eqref{est_comm3} and \eqref{est_comm4}, we also have
\begin{equation}\label{Jordan_s5_2}
    \begin{split}
        \n{h_\gamma^{-2} &\psi_{\rrr,\gamma}(x)^2-\Xi_\gamma(x)^2}=\n{h_\gamma^{-2}\psi_{\rrr,\gamma}(x)^2-h_\gamma^{-1}\psi_{\rrr,\gamma}(x)h_\gamma^{-1}\psi_{\rrr,\gamma}(x)}\\
        & \leq\n{h_\gamma^{-1}}\cdot \n{h_\gamma^{-1}\psi_{\rrr,\gamma}(x)-\psi_{\rrr,\gamma}(x)h_\gamma^{-1}}\cdot \n{\psi_{\rrr,\gamma}(x)}\\
        & \lesssim \n{x}^2\cdot\left\{\begin{array}{ll}
    (C^2(K+2)^5+2C)\n{\psi}\e^{1-3\gamma} & \mbox{under }(\mathsf{H}_1)\\
    D\e^{-2\gamma}\Big\{\n{\psi}^{391/128}\mathcal{O}(\e^{1/256})+\n{\psi}^{31/8}\mathcal{O}(\e^{1/16})\Big\} & \mbox{under }(\mathsf{H}_2).
    \end{array}\right.
    \end{split}
\end{equation}
Now, it is convenient to assume that $\n{\psi}\geq\e^{15/46}$ (otherwise the assertion holds with $\psi_\sss=\psi$ and $\psi_\rrr=0$), as it implies that $\n{\psi}^{23/8}\mathcal{O}(\e^{1/16})$ is majorized by $\n{\psi}^{391/128}\mathcal{O}(\e^{1/256})$. Combining \eqref{Jordan_s5_1} and \eqref{Jordan_s5_2}, and recalling formula \eqref{xioz}, we obtain
\begin{equation}\label{Xi_Jordan}
    \n{\,\Xi_\gamma(x)^2-\Xi_\gamma(x^2)} \lesssim \Delta(\e,\n{\psi})\n{x}^2,
\end{equation}
where
\begin{equation}\label{Delta_formula}
    \Delta(s,t)\coloneqq \left\{\begin{array}{ll}
    (C^2(K+2)^5+10C)ts^{1-3\gamma}+16s^{1-2\gamma} & \mbox{under }(\mathsf{H}_1)\\
    Ds^{-2\gamma}\big\{t^{391/128}\mathcal{O}(s^{1/256})+t^{31/8}\mathcal{O}(s^{1/16})\big\} & \mbox{under }(\mathsf{H}_2).
    \end{array}\right.
\end{equation}
Here, once again, we have possibly changed the value of $D$, whereas $C$ remained the same.

Observe also that for any $x,y\in\Aa_+$ with $xy=0$, we have
$$
\n{\,\Xi_\gamma(x)\Xi_\gamma(y)}\leq \n{h_\gamma^{-1}}\!\cdot\! \n{h_\gamma^{-1}\psi_{\rrr,\gamma}(x)-\psi_{\rrr,\gamma}(x)h_\gamma^{-1}}\!\cdot\! \n{\psi_{\rrr,\gamma}(y)}+\n{h_\gamma^{-1}}^2\!\cdot\!\n{\psi_{\rrr,\gamma}(x)\psi_{\rrr,\gamma}(y)},
$$
where the first summand can be estimated as in \eqref{Jordan_s5_2} replacing $\n{x}^2$ by $\n{x}\n{y}$, whereas the second summand satisfies 
\begin{equation*}
    \n{h_\gamma^{-1}}^2\!\cdot\!\n{\psi_{\rrr,\gamma}(x)\psi_{\rrr,\gamma}(y)}\lesssim \n{x}\n{y}\cdot\left\{\begin{array}{ll}
    C\n{\psi}\e^{1-3\gamma}+2\e^{1-2\gamma} & \mbox{under }(\mathsf{H}_1)\\
    D\e^{-2\gamma}\n{\psi}^{23/8}\mathcal{O}(\e^{1/16}) & \mbox{under }(\mathsf{H}_2),
    \end{array}\right.
\end{equation*} 
due to inequality (iv) and formula \eqref{xioz}. Hence, recalling that $\n{\psi}^{23/8}\mathcal{O}(\e^{1/16})$ is majorized by $\n{\psi}^{391/128}\mathcal{O}(\e^{1/256})$, for all $x,y$ as above, we obtain
\begin{equation}\label{Xi_axi}
\n{\,\Xi_\gamma(x)\Xi_\gamma(y)}\lesssim \Delta(\e,\n{\psi})\n{x}\n{y}.
\end{equation}
Notice that the right-hand side of \eqref{delta_sa_Jordan} is also majorized by $\Delta(\e,\n{\psi})$. Using the Jordan decomposition exactly in the same way as at the~end of the~proof of Proposition~\ref{P_almostJordan}, we conclude from \eqref{Xi_Jordan} and \eqref{Xi_axi} that given any $\eta>0$, we can decrease $\beta$ so that
\begin{equation}\label{Xi_conclusion}
\Xi_\gamma\,\,\mbox{ is }\,\,\eJh{\eta+24\Delta(\e,\n{\psi})}
\end{equation}

Now, we will optimize our choice of the parameter $\gamma$ to make all the relevant error estimates as good as possible in terms of their behavior with respect to $\e$. To this end, recall that $\gamma$ was supposed to satisfy $\e^\gamma<\n{h}$ and if this is not true, then \eqref{norm_of_psi} yields $\n{\psi}\leq (K+2)^5\e^\gamma$; in this case we set $\psi_\sss=\psi$, $\psi_\rrr=0$. So, as indicated several lines above and at the very beginning of the proof, the possible estimates on $\n{\psi_{\sss,\gamma}}$ are of order $\e^{15/46}$ and $\e^\gamma$, whereas \eqref{norm_psi_s} gives an~estimate of order $\e^{\gamma/5}$ and, in general, this is the largest term.

Assuming ($\mathsf{H}_1$), we see from \eqref{Xi_conclusion} that $\Xi_\gamma$ is $\eJh{\delta(\e)}$ with an~error $\delta(\e)$ of order $\e^{1-3\gamma}$. Thus, we want to maximize $\frac{1}{5}\gamma$ and $1-3\gamma$ at the~same time, hence we pick $\gamma$ so that $\frac{1}{5}\gamma=1-3\gamma$. Similarly, under hypothesis ($\mathsf{H}_2$), $\Xi_\gamma$ is $\eJh{\delta(\e)}$ with $\delta(\e)$ of order $\e^{1/256-2\gamma}$ and thus we solve the equation $\frac{1}{5}\gamma=\frac{1}{256}-2\gamma$. Therefore, we put
$$
\gamma=\left\{\begin{array}{cl}
\frac{5}{16} & \mbox{under }(\mathsf{H}_1)\\
\frac{5}{2816} & \mbox{under }(\mathsf{H}_2).
\end{array}\right.
$$
Notice that, indeed, our choice is compatible with the required estimate on $\n{\psi_\sss}$, as the inequality $\n{\psi}\leq (K+2)^5\e^\gamma$ implies that ($\mathsf{A}_1$) works with $\psi_\sss=\psi$ (we have either $\frac{1}{5}\gamma=\frac{1}{16}$ or $\frac{1}{5}\gamma=\frac{1}{2816}>0.0003$).

Having specified $\gamma$, we define $\psi_\sss=\psi_{\sss,\gamma}$, $\psi_\rrr=\psi_{\rrr,\gamma}$ and $\mathcal{C}=\mathcal{C}_\gamma$. Notice that ($\mathsf{A}_1$) follows from \eqref{norm_psi_s}, while ($\mathsf{A}_2$) is just condition (i). Assertion ($\mathsf{A}_3$) follows from (iv) and, for a~suitable $\beta>0$, from formula \eqref{Delta_formula} and condition \eqref{Xi_conclusion}. Note that under ($\mathsf{H}_1$), the term $16\e^{1-2\gamma}$ was majorized by $16\n{\psi}\e^{1-3\gamma}$ and slightly increased due to the~appearance of $\eta>0$ in \eqref{Xi_conclusion}, whereas under hypothesis ($\mathsf{H}_2$) we simply majorized both $\e^{1/256-2\gamma}$ and $\e^{1/16-2\gamma}$ by $\e^{0.0003}$.

Finally, assuming that $\psi$ is positive we can apply Corollary~\ref{h_vs_psi_pos}. If inequality (a') holds true, then our assertion is trivial with $\psi_\sss=\psi$. Otherwise, instead of \eqref{norm_of_psi} we can assume that (b') is valid. Notice that \eqref{psi_hpsih} remains true with $4$ in the place of $K$, whereas in \eqref{PGP}--\eqref{I-PGP} we replace the pair $(K,2)$ by $(4,4)$, due to the remark after Corollary~\ref{h_vs_psi}. Consequently, instead of \eqref{norm_psi_s} we obtain $\n{\psi_{\sss,\gamma}}\leq 24\n{\psi}^{3/5}\e^{1/5}+13\n{\psi}^{4/5}\e^{\gamma/5}$ which can be assumed to be at most $37\n{\psi}^{4/5}\e^{1/16}$ (after picking $\gamma=\frac{5}{16}$). Furthermore, in formula \eqref{xioz} we can replace $\n{\psi}$ by $\n{\psi}_+$ and then in \eqref{norm_of_psir}, \eqref{est_comm2}, \eqref{est_comm3}, \eqref{Jordan_s5_2} and \eqref{Delta_formula} we substitute constant $2$ for $(K+2)^5$.
\end{proof}

\begin{remark}
Having at disposal a~stability result for almost Jordan $^\ast$-homomorphisms, it is possible that another choice of $\gamma$ would be better. For example, suppose that any $\eJh{\e}$ map between given \cs-algebras can be approximated by a~Jordan $^\ast$-homomorphism to within $\mathcal{O}(\e^\omega)$. Then, under hypothesis ($\mathsf{H}_1$), such a~homomorphism would lie at distance of order $\e^{\omega(1-3\gamma)}$ from $\Xi_\gamma$ and we should pick $\gamma$ so that $\frac{1}{5}\gamma=\omega(1-3\gamma)$. A~similar modification would take place under hypothesis ($\mathsf{H}_2$). This is actually the main reason (besides trying to make the whole proof more readable) for which we have been working with the parameter $\gamma$ all the way through and picked its right value just at the end of the proof.
\end{remark}

\section{Examples}
\noindent
Assumption ($\mathsf{H}_1$) of Theorem~\ref{L_decomposition} is automatically satisfied if the codomain algebra is a~matrix algebra $\Ma_n(\C)$, as every $n\times n$ matrix has a~minimal polynomial of degree at most $n$. From the proof of Proposition~\ref{alg_prop}, it follows that the constant $C$ appearing in assertion ($\mathsf{A}_3$) can be then defined by
$$
C=\frac{(8M)^{n-1}-1}{M-\frac{1}{8}},
$$
where $M\leq (K+2)^5$, as can be seen from inequality \eqref{norm_of_psi}. If the map in question is positive we have $M\leq 2$, as then we work under assumption (b') from Lemma~\ref{h_vs_psi_pos}. Therefore, the parameter $\delta$ from assertion ($\mathsf{A}_3$) satisfies
\begin{equation}\label{delta_est}
\delta=\mathcal{O}\big(C^2(K+2)^5\big)\n{\psi}\e^{1/16}=\mathcal{O}\big({(64(K+2)^{10})}^n\big)\n{\psi}\e^{1/16}\quad\mbox{as }\,n\to\infty,
\end{equation}
or, in the positive case, 
$$
\delta=\mathcal{O}(256^n)\n{\psi}\e^{1/16}\quad\mbox{as }\,n\to\infty.
$$
A~similar observation applies in the case where the codomain $\Bb$ is an~arbitrary finite-dimensional \cs-algebra, that is, for some $n_1,\ldots,n_k\in\N$ it can be represented as
\begin{equation}\label{B_rep}
\Bb\cong \Ma_{n_1}(\C)\oplus \Ma_{n_2}(\C)\oplus\ldots\oplus\Ma_{n_k}(\C).
\end{equation}
Here, every element of $\Bb$ is algebraic of order at most $N=n_1\cdot\ldots\cdot n_k$, hence estimate \eqref{delta_est} is valid with $N$ instead of $n$.
\begin{corollary}\label{C_M_n}
Let $\Aa$ be any \cs-algebra and $\Bb$ be a~finite-dimensional \cs-algebra given by {\rm \eqref{B_rep}} with $N=n_1\cdot\ldots\cdot n_k$. Assume that $\phi\in \LL(\Aa,\Bb)$ is an~$\edpr{\e}$ map, with some $\e\in (0,1]$. Then, there exists a~corner \cs-subalgebra $\mathcal{C}$ of $\Bb$ and an~operator $\Phi\in\LL(\Aa^\dag,\mathcal{C})$ satisfying 
$$
\n{\phi-\Phi}\leq (6K+7)\Big(\n{\phi}^{4/5}+\frac{1}{32}\n{\phi}^{9/5}\Big)\e^{1/16}+\frac{1}{2}\e
$$
and such that either $\Phi=0$ or $\Phi(1_{\Aa^\dag})$ is invertible in $\mathcal{C}$, in which case the~operator $\Phi(1_{\Aa^\dag})^{-1}\Phi(\,\cdot\,)$ is $\eJh{\delta}$ with
$$
\delta=\mathcal{O}\big({(64(K+2)^{10})}^{N}\big)\max\{1,\n{\phi}^2\}\e^{1/16}.
$$
\end{corollary}
\begin{proof}
Appealing to Lemma \ref{L_almostsa} we obtain a~self-adjoint $\eoz{(\e+\frac{1}{2}\e\n{\phi})}$ map $\psi\in\LL(\Aa,\Bb)$ such that $\n{\phi-\psi}\leq \frac{1}{2}\e$. Using Lemma~\ref{L_unitization} we may extend $\psi$ to a~map $\psi^\dag\in\LL(\Aa^\dag,\Bb)$ which is still self-adjoint and $\eoz{(\e+\frac{1}{2}\e\n{\phi})}$ It remains to apply Theorem~\ref{L_decomposition} to obtain a~suitable decomposition $\psi^\dag=\psi_\sss+\psi_\rrr$ and observe that the operator $\Phi\coloneqq \psi_\rrr$ satisfies
\begin{equation*}
\begin{split}
\n{\phi-\Phi} &\leq \n{\phi-\psi}+\n{\psi_\sss}\leq (6K+7)\n{\psi}^{4/5}\Big(1+\frac{1}{2}\n{\phi}\Big)^{\! 1/16}\e^{1/16}+\frac{1}{2}\e\\
&\leq (6K+7)\n{\phi}^{4/5}\Big(1+\frac{1}{32}\n{\phi}\Big)\e^{1/16}+\frac{1}{2}\e.
\end{split}
\end{equation*}
(The last estimate follows from Bernoulli's inequality and the fact that $\n{\psi}\leq\n{\phi}$.) From Theorem~\ref{L_decomposition} and the discussion above it follows that either $\Phi=0$ or $\Phi(1_{\Aa^\dag})$ is invertible in the range algebra, in which case the~operator $\Phi(1_{\Aa^\dag})^{-1}\Phi(\,\cdot\,)$ is $\eJh{\delta}$ with
\begin{equation*}
\delta\leq\mathcal{O}\big({(64(K+2)^{10})}^{N}\big)\n{\phi}\Big(1+\frac{1}{32}\n{\phi}\Big)\e^{1/16}.\qedhere
\end{equation*}
\end{proof}

If the map in question is positive, there is, of course, no need of applying Lemma~\ref{L_almostsa} and then Theorem~\ref{L_decomposition} yields the following result.
\begin{corollary}\label{main_corollary}
Let $\Aa$ be any \cs-algebra and $\Bb$ be a~finite-dimensional \cs-algebra given by {\rm \eqref{B_rep}} with $N=n_1\cdot\ldots\cdot n_k$. Assume that $\phi\in \LL(\Aa,\Bb)$ is a~positive $\eoz{\e}$ map, with some $\e\in (0,1]$. Then, there exists a~corner \cs-subalgebra $\mathcal{C}$ of $\Bb$ and an~operator $\Phi\in\LL(\Aa^\dag,\mathcal{C})$ satisfying 
$$
\n{\phi-\Phi}\leq 37\n{\phi}^{4/5}\e^{1/16}
$$
and such that either $\Phi=0$ or $\Phi(1_{\Aa^\dag})$ is invertible in $\mathcal{C}$, in which case the~operator $\Phi(1_{\Aa^\dag})^{-1}\Phi(\,\cdot\,)$ is $\eJh{\delta}$ with
$$
\delta=\mathcal{O}\big(256^{N}\big)\n{\phi}\e^{1/16}.
$$
\end{corollary}

Assumption ($\mathsf{H}_2$) of Theorem~\ref{L_decomposition} is automatically satisfied if the~domain is a~commutative \cs-algebra and the map considered is surjective, which leads to another corollary.
\begin{corollary}
Let $X$ be a locally compact Hausdorff space and $\Bb$ be a~\cs-algebra acting nondegenerately on a~Hilbert space $\Hh$. Assume that $\psi\in\LL(C_0(X),\Bb)$ is a~surjective self-adjoint $\eoz{\e}$ map, with some $\e\in (0,1]$. Then, there exists a~corner \cs-subalgebra $\mathcal{C}$ of $\Bb^\dprime$ and an~operator $\Psi\in\LL(C_0(X)^\dag,\mathcal{C})$ satisfying
$$
\n{\psi-\Psi}\leq (6K+7)\n{\psi}^{4/5}\mathcal{O}(\e^{0.0003})
$$
and such that either $\Psi=0$ or $\Psi(1_{C_0(X)^\dag})$ is invertible in $\mathcal{C}$, in which case the~operator $\Psi(1_{C_0(X)^\dag})^{-1}\Psi(\,\cdot\,)$ is $\eJh{\delta}$ with
$$
\delta=D\big(\n{\psi}^{391/128}+\n{\psi}^{31/8}\big)\mathcal{O}(\e^{0.0003}),
$$
where $D$ depends only on the openness index $\mathrm{op}(\psi)$.
\end{corollary}

\bibliographystyle{amsplain}

\end{document}